\documentclass[reqno]{amsproc}
\usepackage[utf8]{inputenc}
\usepackage{amsfonts}
\usepackage{graphicx}
\usepackage{amscd}
\usepackage{amsmath}
\usepackage{amssymb}
\usepackage{latexsym}
\usepackage{hyperref}
\usepackage[all]{xy}
\usepackage{color}
\usepackage{mathrsfs}
\newtheorem{theorem}{\bf Theorem}
\newtheorem{proposition}{\bf Proposition}
\newtheorem{corollary}{\bf Corollary}
\newtheorem{lemma}{\bf Lemma}
\theoremstyle{definition}

\newtheorem{example}{\bf Example}

\newtheorem{remark}{\bf Remark}

{}
\numberwithin{equation}{section}

\newcommand\dm{\mathrm{dm}}
\newcommand\dn{\mathrm{d\mu}}
\newcommand{\g}[2]{\mbox{$\langle #1 ,#2 \rangle$}}
\newcommand\dv{\mathrm{div}}

\newcommand\Lu{\mathscr{L}}
\newcommand\tr{\mathrm{tr}}
\begin{document}

\title[Eigenvalue estimates on pinched Hadamard manifolds]{Eigenvalue estimates of the drifted Cheng--Yau operator on bounded domains in pinched Cartan--Hadamard manifolds}
\author[Júlio C. M. da Fonseca]{Júlio C. M. da Fonseca$^1$}
\author[José N. V. Gomes]{José N. V. Gomes$^2$}
\address{$^1$Centro de Estudos Superiores de Parintins - CESP, Universidade do Estado do Amazonas -UEA, Parintins, Amazonas, Brazil.}
\address{$^{2}$Departamento de Matemática, Centro de Ciências Exatas e de Tecnologia - CCET, Universidade Federal de São Carlos - UFSCAR, São Carlos, São Paulo, Brazil.}
\email{$^1$jcmfonseca@uea.edu.br}
\email{$^2$jnvgomes@ufscar.br (corresponding author)}
\urladdr{$^1$https://www.uea.edu.br}
\urladdr{$^2$https://www.ufscar.br}
\keywords{Bochner-type formula, Rauch-type comparison theorem, Cheng--Yau operator, Cartan--Hadamard manifolds, Universal eigenvalue inequalities, Rigidity results}
\subjclass[2020]{Primary 47A75; Secondary 58J50, 53C20.}

\begin{abstract} 
We establish universal eigenvalue inequalities for a drifted Cheng--Yau operator on bounded domains in pinched Cartan--Hadamard manifolds with Dirichlet boundary conditions. Our approach combines a universal quadratic eigenvalue inequality for operators in divergence form with a Rauch-type comparison theorem and a Bochner-type formula. We evaluate the geometric terms in this quadratic inequality under suitable assumptions on the radial behavior of the coefficient tensor and the drifting function, together with curvature pinching conditions. In the general setting, we establish a Rauch-type comparison theorem for the Cheng--Yau operator and derive two auxiliary estimates from its Bochner-type formula. We then prove integral estimates of independent interest and apply them to obtain the desired eigenvalue inequalities for both the Cheng--Yau operator and the drifted Laplacian, recovering known inequalities for the Laplacian. We also establish a rigidity result for the Cheng--Yau operator on a class of bounded annular domains in pinched Cartan--Hadamard manifolds. In Euclidean space, this result can be applied, for instance, with the potential function of the Gaussian shrinking soliton. The hyperbolic space is treated separately in two settings, one exploiting the conformal relation between its metric and the Euclidean metric, and the other via its warped product structure, reducing the problem to one dimension. We extend the hyperbolic construction of the fundamental gap for the Laplacian to a class of operators in divergence form. Examples are provided to illustrate the scope and applicability of our results.
\end{abstract}

\maketitle

\section{Introduction}
In this article $M^n$, $n\geq2$, is an $n$-dimensional simply connected smooth manifold with a geodesically complete Riemannian metric $\langle,\rangle$ of sectional curvatures satisfying $-\kappa_1^2\leq K\leq-\kappa_2^2$, where $0\leq \kappa_2\leq \kappa_1$ are constants. It has been called a pinched Cartan--Hadamard manifold after the Cartan--Hadamard theorem, which says that a simply connected geodesically complete Riemannian manifold with nonpositive sectional curvatures is diffeomorphic to a Euclidean space. The two key properties of Cartan--Hadamard manifolds are that any two points in $M^n$ lie on a unique geodesic, and that distance functions are everywhere smooth and convex; see, e.g., Bishop and O'Neill~\cite{Bishop-O'Neil} or the book by Petersen~\cite{Petersen}. Consider a bounded domain $\Omega\subset M^n$ with smooth boundary $\partial\Omega$. Let $T$ be a symmetric positive definite $(1,1)$--tensor on $M^n$, and $\eta$ be a smooth function on $M^n$. Due to the boundedness of $\Omega$ there exist two positive constants $\varepsilon$ and $\delta$ such that $\varepsilon\leq\langle TX,X\rangle\leq\delta$ for any unit vector field $X$ on $\Omega$.

Our aim here is to study the following eigenvalue problem with Dirichlet boundary condition:
\begin{equation*}
\left\{
\begin{array}{ccccc}
-\mathscr{L}u &=& \lambda u &\hbox{in} & \Omega,\\
u&=& 0 & \hbox{on} & \partial\Omega,
\end{array}
\right.
\end{equation*}
where
\begin{equation}\label{defLintro}
\mathscr{L}u=\dv (T(\nabla u))-\langle\nabla\eta, T(\nabla u)\rangle.
\end{equation}
Here, $\dv$ stands for the divergence of smooth vector fields, and $\nabla$ for the gradient of smooth functions.

Associated with $\mathscr{L}$ is the Cheng--Yau operator, defined by
\begin{equation}\label{quadrado}
\Box f:=\operatorname{tr}(\nabla^2 f\circ T)=\langle\nabla^2 f,T\rangle,
\end{equation}
where $\nabla^2 f$ denotes the Hessian of a smooth function $f$ on $M^n$, and $T$ is the tensor associated with $\mathscr{L}$. Introduced by Cheng and Yau~\cite{Cheng-Yau} in their study of hypersurfaces with constant scalar curvature, this operator has since become an important tool in geometric analysis, particularly in rigidity and classification problems.

Alencar, Neto and Zhou~\cite{AGD} showed the Bochner-type formula (see~\eqref{BFCYOp}) for the Cheng--Yau operator. Shortly after that, the second author and Miranda~\cite{GomesMiranda} showed that~\eqref{defLintro} can be decomposed into the \emph{$\eta$--divergence tensor} of $T$ and~\eqref{quadrado}, as follows
\begin{equation}\label{L-Box-div}
\mathscr{L} f=\Box f+\langle\dv_{\eta}T,\nabla f\rangle,
\end{equation}
where $\dv_{\eta}T:=\dv T-\langle\nabla\eta,T(\cdot)\rangle$, and $\dv T$ stands for the divergence tensor of $T$ (see Section~\ref{Sec-L}).

For orientable compact Riemannian manifolds, it has been proved in \cite{Cheng-Yau} that the operator $\Box$ is self-adjoint if and only if $T$ is divergence-free, i.e., $\dv T=0$. In this case, Eq.~\eqref{L-Box-div} becomes
\begin{equation}\label{DCYOp}
\mathscr{L} f=\Box f - \langle\nabla\eta,T(\nabla f)\rangle,
\end{equation}
which is a first-order perturbation of the Cheng--Yau operator. We call Eq.~\eqref{DCYOp} a \emph{drifted Cheng--Yau operator} with a \emph{drifting function} $\eta$. In particular, if $\eta$ is constant, then $\mathscr{L} f=\square f$ is the Cheng--Yau operator with $\dv T=0.$

For instance, if $Ric$ stands for the Ricci tensor of $\langle,\rangle$ and $R=\mathrm{tr}(Ric)$, then, it is known that $\dv Ric = \frac{\mathrm{d}R}{2}$ and $\dv(RI)=\mathrm{d}R$, so the Einstein tensor $G:=Ric-\frac{R}{2}I$ is divergence-free, therefore $\Box f=\langle\nabla^2f,G\rangle$ is self-adjoint on compact Riemannian manifolds, and the drifted Cheng--Yau operator $\mathscr{L} f=\langle\nabla^2f,G\rangle - \langle\nabla\eta,G(\nabla f)\rangle$ is likely to have applications in physics, see, e.g., Serre~\cite{Serre}. We highlight that Serre's work deals with divergence-free positive definite symmetric tensors and fluid dynamics; there, the reader can find examples and know where these tensors occur. In the last part of Section~\ref{Sec-L}, we list some geometric examples of such tensors. We refer the reader to Proposition~5.1 in \cite{AGD} or Navarro~\cite{Navarro} for more related discussions.

Let us consider the hyperbolic space $\mathbb{H}^n(-1)$ with constant curvature $-1$, namely, the open half-space $x_n >0$ with its standard metric $g_{ij}=x_n^{-2}\delta_{ij}$. It is known that $r(x_1,\ldots,x_n)=\ln x_n$ works as a distance function on $\mathbb{H}^n(-1)$. This fact has been our motivation for the first theorem of the paper, in which we prove a universal quadratic inequality for the eigenvalues of $\mathscr{L}$ on any bounded domain in hyperbolic space $\mathbb{H}^n(-\kappa^2)$ with constant curvature $-\kappa^2$ in the upper half-space model. We observe that the hyperbolic space is treated in an independent way in this first result.

In what follows, a smooth function $f$ is called radially constant if its radial derivative vanishes: $\partial_r f=0$, where $r$ is a distance function.

\begin{theorem}\label{propCYT-c3}
Let $\lambda_i$ be the $i$-th eigenvalue of $\mathscr{L}$ on a bounded domain $\Omega\subset \mathbb{H}^n(-\kappa^2)$ with the Dirichlet boundary condition. If the drifting function $\eta$ is radially constant and $T(\nabla\ln x_n)=\psi \nabla\ln x_n$, for some radially constant function $\psi$. Then, for any positive integer $k$, we have
\begin{eqnarray*}
\displaystyle\sum_{i=1}^k(\lambda_{k+1}-\lambda_i)^2 &\leq&\frac{1}{\varepsilon}\sum_{i=1}^k(\lambda_{k+1} -\lambda_i)\Big(\frac{4\delta^2}{\varepsilon}\lambda_i-(n-1)^2\varepsilon^2\kappa^2\Big).
\end{eqnarray*}
Moreover, the first eigenvalue satisfies $\lambda_1(\Omega)\geq\frac{\varepsilon}{4\delta^2}(n-1)^2\varepsilon^2\kappa^2>0.$
\end{theorem}

Theorem~\ref{propCYT-c3} recovers the hyperbolic Laplacian eigenvalue inequality of Cheng and Yang~\cite[Theorem~1.2]{Cheng-Yang-III}. The present extension allows for a positive-definite coefficient tensor and a nonconstant drifting function, subject to the stated invariance assumptions. In particular, choosing the identity tensor yields an eigenvalue estimate for the drifted Laplacian under the same assumption on the drifting function.

Now, we observe that if $T\partial_r=\psi\partial_r$, then $(\nabla_{\partial_r}T)\partial_r=\nabla_{\partial_r}T\partial_r=\partial_r\psi\partial_r,$ since $\nabla_{\partial_r}\partial_r=0$. Hence, $\psi$ is radially constant if and only if $\nabla_{\partial_r}T\partial_r=0$, i.e., $T\partial_r$ is a parallel ﬁeld for a smooth distance function $r$. We say that $T$ is radially parallel if $\nabla_{\partial r}T$ is null.

For our second main result, it is convenient to consider the next constant, which depends on dimension and how the tensor $T$ is bounded on the domain:
\begin{equation}\label{CTE-a}
a(n,\varepsilon,\delta):=-(n-1)^2\varepsilon^2+2(n-1)\delta^2.
\end{equation}

Notice that, for the case of $T=I$, we have $\varepsilon=\delta=1$, then $a(n,\varepsilon,\delta)>0$ for $n=2$, and $a(n,\varepsilon,\delta)\leq0$ for $n\geq3$.
\begin{theorem}\label{thmA3-c3}  
Let $\lambda_i$ be the $i$-th eigenvalue of the drifted Cheng--Yau operator with a drifting function $\eta$ on a bounded domain $\Omega\subset M^n$ with the Dirichlet boundary condition. Fix an origin $o\in M^n\backslash\overline{\Omega}$, and let $r(x)$ be the distance function from $o$. If $T$ is radially parallel and it has $\partial_r$ as an eigenvector, then, for any positive integer $k$, we have:
\begin{enumerate}
\item\label{item1-Thm2} For $a(n,\varepsilon,\delta)\leq0$, 
\begin{eqnarray*}
\sum_{i=1}^k(\lambda_{k+1}-\lambda_i)^2 &\leq&\frac{1}{\varepsilon}\sum_{i=1}^k(\lambda_{k+1} -\lambda_i)\Big(\frac{4\delta^2}{\varepsilon}\lambda_i -(n-1)^2\varepsilon^2\kappa^2_2\\
&&+2(n-1)(\delta^2\kappa^2_1-\varepsilon^2\kappa^2_2)+ 2C_0(n-1)\big(\kappa_1+\frac{1}{d}\big)+C_1\Big).
\end{eqnarray*}
\item\label{item2-Thm2} For $a(n,\varepsilon,\delta)>0$, 
\begin{align*}
\sum_{i=1}^k(\lambda_{k+1}-\lambda_i)^2 \leq&\frac{1}{\varepsilon}\sum_{i=1}^k(\lambda_{k+1} -\lambda_i)\Big(\frac{4\delta^2}{\varepsilon}\lambda_i -(n-1)^2\varepsilon^2\kappa^2_2\\
+&2(n-1)(\delta^2\kappa^2_1-\varepsilon^2\kappa^2_2)+ 2C_0(n-1)(\kappa_1+\frac{1}{d})+C_1 + \frac{a(n,\varepsilon,\delta)}{d^2}\Big).
\end{align*}
\end{enumerate}
Here $d=dist(\Omega,o),$ $C_0=\delta^2\max_{\bar{\Omega}}|\dot{\eta}|$ and $C_1=\delta^2 \max_{\bar{\Omega}}(2\ddot{\eta}-\dot{\eta}^2)$, whereas $\dot{\eta}$ and $\ddot{\eta}$ stand for the first and second radial derivatives of $\eta(x)$, respectively.
\end{theorem}

Since $(M^n,\langle,\rangle)$ is a geodesically complete Riemannian manifold, there exists a sequence of compact subsets $K_j\subset M^n$ with $K_j\subset int{K_{j+1}}$ and $\bigcup_jK_j=M^n$ such that if $q_j\not\in K_j$, then $r(x,q_j)\rightarrow \infty$. As $\Omega$ is bounded, we can assume $\Omega\subset K_{j}$ for some $j$, so that we can define $r(x)$ on $\Omega$ from $o=q_j\in M^n\backslash K_j$, for large enough $j,$ such that both $a(n,\varepsilon,\delta)/d^2$ and $1/d$ are small enough. Moreover, we can take $C_0=\delta^2\max_{\bar{\Omega}}|\nabla\eta|$ and $C_1=\delta^2\max_{\bar{\Omega}}(2|\nabla^2\eta|+|\nabla\eta|^2)$, which do not depend on the distance function. In this context, it is most convenient to take $T=I$ in Theorem~\ref{thmA3-c3} to get the next result.

\begin{corollary}\label{thmA3-c3-1}
Let $\lambda_i$ be the $i$-th eigenvalue of the drifted Laplacian with a drifting function $\eta$ on a bounded domain $\Omega\subset M^n$ with the Dirichlet boundary condition. For any positive integer $k$, we have
\begin{eqnarray*}
\sum_{i=1}^k(\lambda_{k+1}-\lambda_i)^2 &\leq&\sum_{i=1}^k(\lambda_{k+1} -\lambda_i)\Big(4\lambda_i -(n-1)^2\kappa^2_2+2(n-1)(\kappa^2_1-\kappa^2_2)\\
&&+ 2C_0(n-1)\kappa_1+C_1\Big),
\end{eqnarray*}
where $C_0=\max_{\bar{\Omega}}|\nabla\eta|$ and $C_1=\max_{\bar{\Omega}}(2|\nabla^2\eta|+|\nabla\eta|^2).$ Moreover, the first eigenvalue satisfies
\begin{eqnarray*}
\lambda_1(\Omega)\geq\frac{1}{4}\big((n-1)^2\kappa^2_2-2(n-1)(\kappa^2_1-\kappa^2_2) - 2C_0(n-1)\kappa_1 - C_1\big).
\end{eqnarray*}
\end{corollary}

Corollary~\ref{thmA3-c3-1} recovers the Laplacian inequality of Chen, Zheng and Lu~\cite[Theorem~1.1]{CHEN-TAO ZHENG} in dimensions at least three. Note that the present argument also yields the stated estimate in dimension two. For the Cheng--Yau operator case, the inequalities in Theorem~\ref{thmA3-c3} do not depend on the constants $C_0$ and $C_1$, since we can take $\eta$ to be a constant. In the case of the drifted Cheng--Yau operator with the radial drifting function given by $\eta(x)=-2\ln(r(x))$, these inequalities do not depend on the constant $C_1$, since $\eta$ solves $2\ddot{\eta}-\dot{\eta}^2=0$. 

Corollary~\ref{Cor-McKean} below also follows directly from Theorem~\ref{thmA3-c3} and recovers McKean's lower bound~\cite{McKean}. We do not claim an improvement over McKean's classical estimate, which states that if $M^n$ is a Cartan--Hadamard manifold with sectional curvatures satisfying $K\leq-\kappa^2$ for some constant $\kappa>0$, then the spectrum of the Laplacian on $M^n$ lies in $\big[(n-1)^2\kappa^2/4,+\infty\big)$. This lower bound is sharp on the hyperbolic space $\mathbb{H}^n(-\kappa^2)$; see also Cheng~\cite{S.Y.Cheng-I}. In Pinsky~\cite {PinskyI}, we found more accurate bounds in the case of surfaces. Thus, although Corollary~\ref{Cor-McKean} does not establish sharpness, it identifies a natural pinching condition under which the lower bound obtained is strictly positive, thereby clarifying the geometric regime in which the estimate provides nontrivial information.

\begin{corollary}\label{Cor-McKean}
The first eigenvalue $\lambda_1$ of the Laplacian on a bounded domain $\Omega\subset M^n$ with the Dirichlet boundary condition satisfies
\begin{eqnarray*}
\lambda_1(\Omega)\geq \frac{(n-1)}{4}^2\kappa^2_2 - \frac{n-1}{2}(\kappa^2_1-\kappa^2_2).
\end{eqnarray*}
Moreover, this lower bound is positive for $0 < \kappa_2\leq \kappa_1$ and $\sqrt{n+1}\kappa_2>\sqrt{2}\kappa_1.$
\end{corollary}

The next corollary concerns the fixed-index Dirichlet eigenvalues of domains containing hyperbolic balls of arbitrarily large radius. Their convergence to the bottom of the hyperbolic spectrum was established by Cheng and Yang~\cite{Cheng-Yang-III} in the unit-curvature case, and the corresponding result for other negative constant curvatures follows by rescaling. Accordingly, Corollary~\ref{newAproach} below is presented not as a new spectral limit, but as an application of our general eigenvalue inequalities and a recovery of this classical result.

Section~\ref{Sec5} further elucidates the geometric mechanism underlying this application through a radial variational argument made possible by the warped product structure of the hyperbolic space. We provide a complete proof of Cheng and Yang's result by combining Theorem~\ref{propCYT-c3} with the approach developed by the second author and Marrocos~\cite{Marrocos-Gomes}, in which warped products play a central role. That work provided a partial answer to a question posed by Zelditch in 1990~\cite{zeld} concerning the generic multiplicity of Laplace eigenvalues on principal bundles. It also contributed to the broader problem of characterizing manifolds with generic $G$-simple spectrum for the real Laplace--Beltrami operator, listed as Problem 42 in Yau's collection of open problems (see Greene and Yau~\cite{yau}).

The contribution of the present section is therefore not the spectral limit itself, but an application of Theorem~\ref{propCYT-c3}, which extends the scope of the underlying eigenvalue methods. The hyperbolic application demonstrates how these general estimates recover classical spectral behavior while connecting the analysis of the Cheng--Yau operator with radial variational techniques and the geometry of warped products.

\begin{corollary}\label{newAproach}
The $i$-th eigenvalue $\lambda_i$ of the Laplacian on a bounded domain $\Omega\subset \mathbb{H}^n(-\kappa^2)$ with the Dirichlet boundary condition satisfies
\begin{equation*}
\lim_{\Omega\to\mathbb{H}^n(-\kappa^2)}\lambda_i(\Omega)=\frac{(n-1)^2}{4}\kappa^2.
\end{equation*}
\end{corollary}

We identify the inequalities in Theorem~\ref{thmA3-c3} as the most appropriate tool for the applications of our results. Note that the appearance of the constants $C_0$ and $C_1$ is natural and motivates us to ask the following question:

\vspace{0.2cm}
\emph{Under what conditions are the eigenvalue estimates obtained in Theorem~\ref{thmA3-c3} independent of the constants $C_0$ and $C_1$ for a nontrivial drifting function $\eta$?}
\vspace{0.2cm}

A radially constant drifting function gives an immediate instance of this phenomenon: the radial derivatives entering the bound vanish. Corollary~\ref{Rig1} therefore concerns the form of the eigenvalue estimate. In this discussion, invariance under a first-order perturbation means that the same bound applies; it does not assert equality of spectra. The annular application in Corollary~\ref{InvCor} is interpreted in the same sense.

\begin{corollary}\label{Rig1}
Let $\lambda_i$ be the $i$-th eigenvalue of the drifted Cheng--Yau operator with a radially constant drifting function on a bounded domain $\Omega\subset M^n$ with the Dirichlet boundary condition. Fix an origin $o\in M^n\backslash\overline{\Omega}$, and let $r(x)$ be the distance function from $o$. If $T$ is radially parallel and it has $\partial_r$ as an eigenvector, then all estimates for the sequence of eigenvalues $(\lambda_i)$ that we can obtain from Theorem~\ref{thmA3-c3} do not depend on this drifting function.
\end{corollary}

In the case of radial drifting functions, we again fix an origin $o\in M^n$ to define the distance function $r(x)$ from $o$, and let us consider the bounded annular domain 
\begin{equation}\label{Omega}
\Omega=\Big\{x\in M^n; \frac{2(n-1)(\kappa_1+\alpha) R +2}{c} < r(x)^2 < R^2 \Big\},
\end{equation}
where $\alpha=1/d$, $cR>(n-1)(\kappa_1+\alpha)+\sqrt{(n-1)^2(\kappa_1+\alpha)^2+2c}$ and $c$ are positive constants. To obtain such a constant $R$, it is enough to consider the quadratic polynomial $\mathscr{P}(x)= cx^2-2(n-1)(\kappa_1+\alpha)x-2.$ So, for the radial drifting function $\eta(x)=\frac{c}{2}r(x)^2$ on $\Omega$, we have $C_0=\delta^2cR$ and
\begin{equation*}
\max_{\bar\Omega}\big(2c - c^2 r(x)^2 \big) = 2c - c^2 \min_{\bar{\Omega}}r(x)^2=2c-(2c(n-1)(\kappa_1+\alpha) R+2c).
\end{equation*}
Thus, it is null the expression $2C_0(n-1)(\kappa_1+\frac{1}{d})+C_1$ that appears in Theorem~\ref{thmA3-c3}. We therefore obtain the following rigidity result for the Cheng--Yau operator on bounded annular domains as described above.

\begin{corollary}\label{InvCor}
Let $\lambda_i$ be the $i$-th eigenvalue of the drifted Cheng--Yau operator with the drifting function $\eta(x)=\frac{c}{2}r(x)^2$ on the bounded annular domain defined as in~\eqref{Omega} with the Dirichlet boundary condition. Fix an origin $o\in M^n\backslash\overline{\Omega}$, and let $r(x)$ be the distance function from $o$. If $T$ is radially parallel and it has $\partial_r$ as an eigenvector, then all estimates for the sequence of eigenvalues $(\lambda_i)$ that we can obtain from Theorem~\ref{thmA3-c3} do not depend on the constants $C_0$ and $C_1$.
\end{corollary}

To obtain an application of Corollary~\ref{InvCor}, the reader can consider the drifted Cheng--Yau operator with the drifting function $\eta(x)=\frac{\lambda}{2}|x|^2$ on bounded annular domains defined as in \eqref{Omega} in Gaussian shrinking soliton $(\mathbb{R}^n,\delta_{ij}, \frac{\lambda}{2}|x|^2)$. 

Section~\ref{secaogap} investigates the fundamental gap for the operator $\dv(\varphi\nabla u)$, where $\varepsilon\leq\varphi\leq\delta$ and $\varphi$ is radially constant on each domain in a family of convex domains in the hyperbolic plane. The domains and the separation-of-variables method are adapted from Bourni et al.~\cite{Bourni}, who studied the Laplacian. Our contribution is to extend their fundamental-gap analysis to this class of operators in divergence form.

We conclude this introduction by outlining the main contributions of this article and their relation to previous work. The Bochner-type formulas used here were established in~\cite{AGD,GomesMiranda}, and the universal quadratic eigenvalue inequality for the operator $\mathscr{L}$ was derived in~\cite{GomesMiranda}. The main contribution of the present work is to estimate the geometric terms in this quadratic inequality under suitable assumptions on the radial behavior of the coefficient tensor $T$ and the drifting function $\eta$, together with curvature pinching conditions. To this end, Proposition~\ref{Lem-LHCThm} establishes a Rauch-type comparison theorem for the Cheng--Yau operator and provides two auxiliary estimates required for the associated Bochner-type formula. Propositions~\ref{lemmaboxestimate} and~\ref{Prop-C0} then establish the integral estimates needed to prove Theorems~\ref{propCYT-c3} and~\ref{thmA3-c3}. The first theorem recovers the hyperbolic Laplacian eigenvalue inequality of Cheng and Yang~\cite{Cheng-Yang-III} as a special case, while extending it to operators with a positive-definite coefficient tensor and a nonconstant drifting function under the stated invariance assumptions. The second extends the distance-function method for the Laplacian used by Chen, Zheng, and Lu~\cite{CHEN-TAO ZHENG} to specific classes of $\mathscr{L}$ under geometric assumptions on $T$ and $\eta$. The examples illustrate the scope and applicability of our results.

\section{Preliminaries}\label{Sec-L}
In this section, we fix notation, comment on facts that will be used in our proofs, and list without proof all the main formulas which will be appropriate for us. The material we summarize here is known; Gomes and Miranda~\cite{GomesMiranda} is a reference for it.

Let $(M^n,\langle,\rangle)$ be an $(n\geq2)$-dimensional Riemannian manifold and $T$ be a $(0,2)$--tensor on $M^n$. Throughout the paper, we will be constantly using the identification of $T$ with its associated $(1,1)$--tensor by the equation
\begin{equation*}\label{AAAC-c3}
\langle T(X),Y\rangle=T(X,Y).
\end{equation*}
In particular, the metric tensor $\langle,\rangle$ will be identified with the identity $I$ in $\mathfrak{X}(M).$

Let $\{e_1,\ldots,e_n\}$ be an orthonormal basis in $T_pM$, and $S$ be a $(1,1)$--tensor with adjoint $S^*$. Recall that the {\it Hilbert-Schmidt inner product} is defined as
\begin{equation*}
\langle T,S\rangle:=\mathrm{tr}(TS^*) = \sum_{i=1}^{n}\langle T(e_i),S(e_i)\rangle.
\end{equation*}

The divergence of a $(1,1)$--tensor $T$ is defined as the $(0,1)$--tensor
\begin{equation*}
(\dv T)(v)(p)=\mathrm{tr}(w\mapsto(\nabla_{w}T)(v)(p)),
\end{equation*}
where $p\in M^n$, $v,w\in T_{p}M,$ $\nabla$ stands for the covariant derivative of $T$ and $\mathrm{tr}$ is the trace operator calculated in the metric $\langle,\rangle.$ Note that we can use the identification $(\dv T)(v)=\langle\dv T,v\rangle.$

In this article, the manifold $(M^n,\langle,\rangle)$ is assumed to be complete, and $\Omega\subset M^n$ is a bounded domain assumed to be connected and with smooth boundary $\partial\Omega$. We are using the weighted measure $\dm=e^{-\eta}dvol_\Omega$, for some smooth function $\eta$. If $\nu$ is the outward normal vector field on the boundary $\partial\Omega$, then the divergence theorem remains true in the form
\begin{equation*}
\int_{\Omega}\mathscr{L}f\dm=\int_{\partial\Omega}T(\nabla f,\nu)\dn,
\end{equation*}
where $\dn=e^{-\eta}dvol_{\partial\Omega}$ is the weighted measure on the boundary induced by $\nu$. Thus, the ``integration by parts'' formula is
\begin{equation*}
\int_{\Omega}\ell\mathscr{L}f\dm=-\int_{\Omega}T(\nabla\ell,\nabla f)\dm+\int_{\partial\Omega}\ell T(\nabla f,\nu)\dn,
\end{equation*}
from which we conclude that $\mathscr{L}$ is a self-adjoint operator in the Hilbert space $\mathcal{H}_0^1(\Omega,\dm)$. Thus, the Dirichlet eigenvalue problem
\begin{equation}\label{PD-c3}
\left\{
\begin{array}{ccccc}
-\mathscr{L}u &=& \lambda u & \hbox{in} & \Omega,\\
u&=&0 & \hbox{on} & \partial\Omega
\end{array}
\right.
\end{equation}
has a real and discrete spectrum $0<\lambda_{1}\leq\lambda_{2}\leq\lambda_3\leq\cdots\to +\infty$, where each $\lambda_i$
is repeated according to its multiplicity. Eigenspaces belonging to distinct eigenvalues are orthogonal in $L^2(\Omega,\dm)$, which is the direct sum of all the eigenspaces. We refer to the dimension of each eigenspace as the multiplicity of the eigenvalue. Moreover, for any eigenfunction $u_i$ we have
\begin{eqnarray*}
\lambda_i=-\int_{\Omega}u_i\mathscr{L}u_i\dm=\int_{\Omega}T(\nabla
u_i,\nabla u_i)\dm.
\end{eqnarray*}

The \emph{Bochner-type formula for the more general expression of $\mathscr{L}$} is given by
\begin{equation}\label{BF}
\frac{1}{2}\mathscr{L}(|\nabla f|^2)=\langle\nabla(\mathscr{L}f),\nabla f\rangle+ R_{\eta,T}(\nabla f,\nabla f) +\langle\nabla^2f,\nabla^2f\circ T\rangle-\langle\nabla^2f,\nabla_{\nabla f}T\rangle
\end{equation}
where $R_T(X,Y)=\mathrm{tr}(Z\mapsto T\circ R(X,Z)Y)$, $R_{\eta,T}:=R_T-\nabla(\dv_\eta T)^\sharp$ and $R(X,Z)Y$  is the curvature tensor of the Riemannian metric $\langle,\rangle$ on $M^n$. Here, the map $(\cdot)^\sharp$ stands for the musical isomorphism.

In particular, by taking $\eta$ to be constant and $T$ divergence-free in \eqref{BF}, we get the \emph{Bochner-type formula for the Cheng--Yau operator} as follows
\begin{equation}\label{BFCYOp}
\frac{1}{2}\square (|\nabla f|^2)=\g{\nabla\square f}{\nabla f}+R_T(\nabla f,\nabla f)+\g{\nabla^2f}{\nabla^2f\circ T}-\g{\nabla^2f}{\nabla_{\nabla f}T}.
\end{equation}

Two special cases of divergence-free positive definite symmetric tensors on $\big(\mathbb{H}^n(-1),\langle,\rangle\big)$ are $T=-Ric$ and $T=-S$, where $S=\frac{1}{n-2}\Big(Ric-\frac{R}{2(n-1)}\langle,\rangle\Big)$ is the Schouten tensor of $\langle,\rangle$, for $n\geq3$. More generally, these special tensors can be considered on pinched Cartan--Hadamard Einstein manifolds $(M^n,\langle,\rangle)$, since $Ric=\frac{R}{n}\langle,\rangle$ and by Schur's lemma the scalar curvature of Einstein manifolds of dimension $n\geq 3$ must be constant. We highlight that geodesically complete noncompact Einstein manifolds with $Ric=-(n-1)\langle,\rangle$ have a very special behavior at infinity (see Gicquaud, Ji and Shi~\cite{RJS}). Another example is obtained from $\hat{S}= S-\tr(S)\langle,\rangle$ on pinched Cartan--Hadamard manifolds $(M^n,\langle,\rangle)$, $n\geq3$, in this case, we have $\dv\hat{S}=0$, since $\dv S=d\tr(S)$.

\section{Auxiliary results}\label{Sec-Aux-Res}
The key to proving the two main theorems of this paper relies on Proposition~\ref{Prop1GM}, which is a slight modification of Proposition~1 in~\cite{GomesMiranda}. With this in mind, we establish the necessary tools to work with the operator defined in Eq.~\eqref{defLintro}, which enable us to obtain more general results. We believe that such tools are of independent interest. 

An important lemma that will be used in the proof of Proposition~\ref{Lem-LHCThm} is a known result in comparison geometry. Its proof can be found in~\cite{Petersen}.
\begin{lemma}\label{RC}[Rauch Comparison] Assume that $(M^n, \langle,\rangle)$ satisfies $c\leq K\leq C$. If $\langle,\rangle=dr^2+g_r$ represents the metric in polar coordinates, then
\begin{equation*}
\frac{sn'_C(r)}{sn_C(r)}g_r\leq\nabla^2r\leq\frac{sn'_c(r)}{sn_c(r)}g_r,
\end{equation*}
where $sn_\kappa(r)$ denotes the unique solution to
\begin{equation*}
\ddot{x}(r) + \kappa\cdot x(r) = 0 \quad \hbox{with}\quad x(0)=0\quad \hbox{and}\quad \dot{x}(0)=1.
\end{equation*}
In particular, $\frac{sn'_\kappa(r)}{sn_\kappa(r)}=\sqrt{-\kappa}\frac{\cosh(\sqrt{-\kappa}r)}{\sinh(\sqrt{-\kappa}r)}$ for $\kappa<0$, and $\frac{sn'_\kappa(r)}{sn_\kappa(r)}=\frac{1}{r}$ for $\kappa=0$.
\end{lemma}

In what follows, $\Omega\subset M^n$, $n\geq2$, is a bounded domain in an $n$-dimensional pinched Cartan--Hadamard manifold as described in our introduction.

Our first proposition establishes a \emph{Rauch-type comparison theorem} for the Cheng--Yau operator and derives two estimates that will be used in the Bochner-type formula for this operator.
 
\begin{proposition}\label{Lem-LHCThm}
Fix an origin $o\in M^n\backslash\overline{\Omega}$, and let $r(x)$ be the distance function from $o$. Let $T$ be a symmetric positive definite $(1,1)$--tensor on $M^n$ such that $\partial_r$ is an eigenvector of $T$. Then, for $C=-\kappa_2^2$ and $c=-\kappa_1^2$, the following holds on $\Omega$:
\begin{enumerate}
\item $(n-1)\varepsilon \frac{sn'_C(r)}{sn_C(r)}\leq \Box r \leq (n-1)\delta \frac{sn'_c(r)}{sn_c(r)}.$
\item $\g{\nabla^2r}{\nabla^2r\circ T}\leq(n-1)\delta\Big(\frac{sn'_c(r)}{sn_c(r)}\Big)^2.$
\item $R_T(\partial_r,\partial_r)\leq -\varepsilon(n-1)\kappa_2^2.$
\end{enumerate}
\end{proposition}

\begin{proof}
Take $x\in\Omega$ and complete $\partial_r$ to an orthonormal basis $\{e_1,\ldots,e_n=\partial_r\}$ for $T_xM$ such that $Te_i=\psi_ie_i$. Note that $\varepsilon\leq\psi_i\leq\delta$ on $\Omega$, for $i=1,\ldots,n.$ Thus,
\begin{equation*}
\Box r = \g{\nabla^2r}{T}=\sum_{i=1}^{n}\g{\nabla^2r(e_i)}{T(e_i)}=\sum_{i=1}^{n-1}\psi_i\g{\nabla^2r(e_i)}{e_i}.
\end{equation*}
Since $r(x)$ is a convex function, i.e., its Hessian is positive semidefinite, one has
\begin{equation*}
\varepsilon \Delta r=\varepsilon \sum_{i=1}^{n-1}\g{\nabla^2r(e_i)}{e_i} \leq \Box r \leq \delta \sum_{i=1}^{n-1}\g{\nabla^2r(e_i)}{e_i}=\delta\Delta r.
\end{equation*}
Hence, the first assertion follows from Lemma~\ref{RC}. To prove the second assertion, we compute
\begin{equation*}
\g{\nabla^2r}{\nabla^2r\circ T}=\sum_{i=1}^{n}\g{\nabla^2r(e_i)}{\nabla^2 r\circ T(e_i)}=\sum_{i=1}^{n-1}\psi_i\g{\nabla^2r(e_i)}{\nabla^2r(e_i)}.
\end{equation*}
Thus, again we use that $r(x)$ is a convex function and that $\psi_i\leq\delta$ to obtain 
\begin{equation*}
\g{\nabla^2r}{\nabla^2r\circ T}\leq \delta |\nabla^2 r |^2.
\end{equation*}
Now, we use Lemma~\ref{RC} to get the second assertion. For the third assertion, we have
\begin{eqnarray*}
R_T(\partial_r,\partial_r)&=&\sum_{i=1}^{n-1}\g{R(e_i,\partial_r)\partial_r}{T(e_i)}=\sum_{i=1}^{n-1}\psi_i\g{R(e_i,\partial_r)\partial_r}{e_i}\\ &\leq&\sum_{i=1}^{n-1}-\psi_i\kappa_2^2 \leq -\varepsilon(n-1)\kappa_2^2.
\end{eqnarray*}
This completes the proof of the proposition.
\end{proof}

For the second proposition, it is convenient to consider the constant defined in Eq.~\eqref{CTE-a}.

\begin{proposition}\label{lemmaboxestimate}
Fix an origin $o\in M^n\backslash\overline{\Omega}$, and let $r(x)$ be the distance function from $o$. Let $T$ be a symmetric positive definite $(1,1)$--tensor on $M^n$ such that $\partial_r$ is an eigenvector of $T$, and $u_i$ be an $L^2(\Omega,\dm)$--normalized function. If $T$ is radially parallel, then:
\begin{enumerate}
\item For $a(n,\varepsilon,\delta)\leq0$, it is true that
\begin{equation*}
\int_{\Omega}u^2_i\Big(-(\square r)^2-2\g{\nabla\square r}{T\partial_r}\Big)\dm\leq -(n-1)^2\varepsilon^2\kappa^2_2+2(n-1)(\delta^2\kappa^2_1-\varepsilon^2\kappa^2_2).
\end{equation*}
\item For $a(n,\varepsilon,\delta)>0$, we consider $d=dist(\Omega,o)$  so that
\begin{equation*}
\int_{\Omega}u^2_i\Big(-(\square r)^2-2\g{\nabla\square r}{T\partial_r}\Big)\dm\leq -(n-1)^2\varepsilon^2\kappa^2_2+2(n-1)(\delta^2\kappa^2_1-\varepsilon^2\kappa^2_2) + \frac{a(n,\varepsilon,\delta)}{d^2}.
\end{equation*}
\end{enumerate}
\end{proposition}

\begin{proof}
We begin by estimating the expression $-(\square r)^2-2\g{\nabla\square r}{T\partial_r}$. For this, we use that $T$ is radially parallel, i.e., $\nabla_{\partial_r}T$ is null, so that, from the Bochner-type formula for the Cheng--Yau operator, we obtain
\begin{equation*}
-\g{\nabla\square r }{\partial_r}=R_T(\partial_r,\partial_r)+\g{\nabla^2r}{\nabla^2r\circ T}.
\end{equation*}
Since $T\partial_r=\psi_n\partial_r$, from the second and third assertions of Proposition~\ref{Lem-LHCThm}, we have
\begin{equation}\label{IPTr}
-\g{\nabla\square r }{T\partial_r}=-\psi_n\g{\nabla\square r }{\partial_r} \leq - (n-1)\varepsilon^2\kappa_2^2 + (n-1)\delta^2\left(\frac{sn'_{-\kappa_1^2}(r)}{sn_{-\kappa_1^2}(r)}\right)^2.
\end{equation}
There are three cases to consider:

\vspace{0.1cm}
\noindent\textbf{(a)} $0 <\kappa_2\leq \kappa_1$ case: Inequality~\eqref{IPTr} becomes
\begin{equation*}
-\g{\nabla\square r }{T\partial_r} \leq (n-1)\delta^2\kappa^2_1\frac{\cosh^2(\kappa_1r)}{\sinh^2(\kappa_1r)} - (n-1)\varepsilon^2\kappa_2^2.
\end{equation*}
So, from the first assertion of Proposition~\ref{Lem-LHCThm}, we estimate the expression
\begin{align*}
&-(\square r)^2-2\g{\nabla\square r}{T\partial_r}\\
&\leq -(n-1)^2\varepsilon^2\kappa^2_2\frac{\cosh^2(\kappa_2r)}{\sinh^2(\kappa_2r)} +2(n-1)\delta^2\kappa^2_1\frac{\cosh^2(\kappa_1r)}{\sinh^2(\kappa_1r)}-2(n-1)\varepsilon^2\kappa_2^2\\
&=-(n-1)^2\varepsilon^2\kappa^2_2-\frac{(n-1)^2\varepsilon^2\kappa^2_2}{\sinh^2(\kappa_2r)}+2(n-1)\delta^2\kappa^2_1+
\frac{2(n-1)\delta^2\kappa_1^2}{\sinh^2(\kappa_1r)}-2(n-1)\varepsilon^2\kappa_2^2\\
&=-(n-1)^2\varepsilon^2\kappa^2_2+2(n-1) (\delta^2\kappa^2_1-\varepsilon^2\kappa^2_2)+\left[-\frac{(n-1)^2\varepsilon^2\kappa^2_2}{\sinh^2(\kappa_2r)}+\frac{2(n-1)\delta^2\kappa^2_1}{\sinh^2(\kappa_1r)}\right].
\end{align*}
Since $0 <\kappa_2\leq \kappa_1$ and $r>0$, we get
\begin{eqnarray*}
\frac{\kappa^2_1}{\sinh^2(\kappa_1r)}\leq \frac{\kappa^2_2}{\sinh^2(\kappa_2r)}.
\end{eqnarray*}
Thus,
\begin{equation}\label{a-case}
-\frac{(n-1)^2\varepsilon^2\kappa^2_2}{\sinh^2(\kappa_2r)}+\frac{2(n-1)\delta^2\kappa^2_1}{\sinh^2(\kappa_1r)}\leq \left(-(n-1)^2\varepsilon^2+2(n-1)\delta^2\right)\frac{\kappa^2_1}{\sinh^2(\kappa_1r)}.
\end{equation}

\vspace{0.1cm}
\noindent\textbf{(b)} $0=\kappa_2 < \kappa_1$ case: Again from the first assertion of Proposition~\ref{Lem-LHCThm} and Inequality~\eqref{IPTr}, we estimate the expression
\begin{align*}
-(\square r)^2-2\g{\nabla\square r}{T\partial_r}
&\leq -\frac{(n-1)^2\varepsilon^2}{r^2}+2(n-1)\delta^2\kappa^2_1\frac{\cosh^2(\kappa_1r)}{\sinh^2(\kappa_1r)}\\
&= 2(n-1)\delta^2\kappa^2_1 +\left[-\frac{(n-1)^2\varepsilon^2}{r^2}+
\frac{2(n-1)\delta^2\kappa_1^2}{\sinh^2(\kappa_1r)}\right].
\end{align*}
Since $0 < \kappa_1$ and $r>0$, we get
\begin{equation}\label{Ineq-sinh-r}
\frac{\kappa^2_1}{\sinh^2(\kappa_1r)}\leq \frac{1}{r^2}.
\end{equation}
Thus,
\begin{equation}\label{b-case}
-\frac{(n-1)^2\varepsilon^2}{r^2}+\frac{2(n-1)\delta^2\kappa^2_1}{\sinh^2(\kappa_1r)}\leq \left(-(n-1)^2\varepsilon^2+2(n-1)\delta^2\right)\frac{\kappa^2_1}{\sinh^2(\kappa_1r)}.
\end{equation}

\vspace{0.1cm}
\noindent\textbf{(c)} $0 =\kappa_2=\kappa_1$ case: Again from the first assertion of Proposition~\ref{Lem-LHCThm} and Inequality~\eqref{IPTr}, we estimate the expression
\begin{equation}\label{c-case}
-(\square r)^2-2\g{\nabla\square r}{T\partial_r} \leq \Big[-(n-1)^2\varepsilon^2+2(n-1)\delta^2\Big]\frac{1}{r^2}.
\end{equation}
Now, with inequalities~\eqref{a-case}, \eqref{b-case} and \eqref{c-case} in mind, we immediately obtain the first integral estimate of the proposition for $a(n,\varepsilon,\delta)\leq0$. For the case of $a(n,\varepsilon,\delta)>0$, first note that
\begin{equation}\label{Est-d}
\int_{\Omega}u_i^2\frac{a(n,\varepsilon,\delta)}{r^2}\dm\leq\frac{a(n,\varepsilon,\delta)}{d^2},
\end{equation}
where $d=dist(\Omega,o)$, since that $0<d\leq r(x)$, for all $x\in\Omega$. We now complete our proof immediately from inequalities \eqref{a-case}, \eqref{Ineq-sinh-r}, \eqref{b-case}, \eqref{c-case} and \eqref{Est-d}.
\end{proof}

\begin{proposition}\label{Prop-C0}
Fix an origin $o\in M^n\backslash\overline{\Omega}$, and let $r(x)$ be the distance function from $o$. Let $T$ be a symmetric positive definite $(1,1)$--tensor on $M^n$ such that $\partial_r$ is an eigenvector of $T$, and $u_i$ be an $L^2(\Omega,\dm)$--normalized  function. Then,
\begin{equation*}
\int_\Omega u^2_i\Box r \g{T\partial_r}{\nabla\eta}\dm\leq C_0(n-1)\big(\kappa_1 + \frac{1}{d}\big),
\end{equation*}
where $d=dist(\Omega,o)$ and $C_0=\delta^2\max_{\bar{\Omega}}|\dot{\eta}|$.
\end{proposition}

\begin{proof}
From the first assertion of Proposition~\ref{Lem-LHCThm}, we estimate
\begin{eqnarray*}
\int_\Omega u^2_i\Box r \g{T\partial_r}{\nabla\eta}\dm &\leq& 
\Big(\int_\Omega u^2_i(\Box r)^2\dm\Big)^{\frac{1}{2}}\Big(\int_\Omega u^2_i\g{T\partial_r}{\nabla\eta}^2\dm\Big)^{\frac{1}{2}}\\
&\leq& \delta \max_{\bar{\Omega}}|\dot{\eta}|\Big(\int_\Omega u^2_i(\Box r)^2\dm\Big)^{\frac{1}{2}}\\
&\leq& C_0(n-1)\Big(\int_\Omega u^2_i\Big(\frac{sn'_{-\kappa^2_1}(r)}{sn_{-\kappa^2_1}(r)}\Big)^2\dm\Big)^{\frac{1}{2}}.
\end{eqnarray*}
There are two cases to consider:

\vspace{0.1cm}
\noindent\textbf{(a)} $0 <\kappa_2\leq \kappa_1$ and $0=\kappa_2 < \kappa_1$ cases:
\begin{eqnarray*}
\int_\Omega u^2_i\Box r \g{T\partial_r}{\nabla\eta}\dm&\leq& C_0(n-1)\kappa_1\Big(\int_\Omega u^2_i\frac{\cosh^2(\kappa_1r)}{\sinh^2(\kappa_1r)}\dm\Big)^{\frac{1}{2}}\\
&\leq& C_0(n-1)\kappa_1+C_0(n-1)\Big(\int_\Omega u^2_i\frac{\kappa_1^2}{\sinh^2(\kappa_1r)}\dm\Big)^{\frac{1}{2}}.\\
&\leq& C_0(n-1)\Big(\kappa_1+\Big(\int_\Omega u^2_i\frac{\kappa_1^2}{\sinh^2(\kappa_1r)}\dm\Big)^{\frac{1}{2}}\Big).
\end{eqnarray*}

\vspace{0.1cm}
\noindent\textbf{(b)} $\kappa_1=\kappa_2=0$ case:
\begin{equation*}
\int_\Omega u^2_i\Box r \g{T\partial_r}{\nabla\eta}\dm \leq C_0(n-1)\Big(\int_\Omega u^2_i\frac{1}{r^2}\dm\Big)^{\frac{1}{2}}.
\end{equation*}
As we argued in the proof of Proposition~\ref{lemmaboxestimate}, we obtain the required integral estimate of the present proposition.
\end{proof}

\section{Proof of Theorems~\ref{propCYT-c3} and \ref{thmA3-c3}}
The next step before giving the proof of the two main theorems is a universal quadratic inequality for the eigenvalues of Problem~\eqref{PD-c3}, which is an essential tool for us.
\begin{proposition}\label{Prop1GM}
Fix an origin $o\in M^n\backslash\overline{\Omega}$, and let $r(x)$ be the distance function from $o$. Let $\lambda_i$ be the $i$-th eigenvalue of Problem~\eqref{PD-c3}, and $u_i$ be its corresponding $L^2(\Omega,\dm)$--normalized eigenfunction. Then, 
\begin{equation*}
\sum_{i=1}^k(\lambda_{k+1}-\lambda_i)^2 \leq\frac{1}{\varepsilon}\sum_{i=1}^k(\lambda_{k+1} -\lambda_i)\Big(\frac{4\delta^2}{\varepsilon}\lambda_i - \int_{\Omega}u_i^2\big((\mathscr{L}r)^2+2\g{T\partial_r}{\nabla \Lu r}\big)\dm\Big).
\end{equation*}
\end{proposition}
\begin{proof}
Proposition~1 in \cite{GomesMiranda} says that
\begin{equation*}
\sum_{i=1}^k(\lambda_{k+1}-\lambda_i)^2\! \int_{\Omega} u_i^2T(\nabla h,\nabla h)\dm \leq\sum_{i=1}^k(\lambda_{k+1} -\lambda_i)\! \int_{\Omega}\big\{u_i\mathscr{L}h+2T(\nabla h,\nabla u_i)\big\}^2\dm
\end{equation*}
for any $h \in \mathcal{C}^3(\Omega)\cap\mathcal{C}^2(\partial \Omega)$. By taking $h=r$, and noting that $\varepsilon \leq T(\partial_r,\partial_r)\leq\delta$, we obtain
\begin{equation}\label{Eq1-Prop1GM}
\varepsilon\sum_{i=1}^k(\lambda_{k+1}-\lambda_i)^2 \leq\sum_{i=1}^k(\lambda_{k+1} -\lambda_i) \int_{\Omega}\big\{u_i\mathscr{L}r+2T(\partial_r,\nabla u_i)\big\}^2\dm.
\end{equation}
The integral in \eqref{Eq1-Prop1GM} is estimated as follows
\begin{align}\label{Aux-Ju1}
\nonumber&\int_{\Omega}\big\{u_i^2(\mathscr{L}r)^2+4u_i\Lu rT(\partial_r,\nabla u_i)+4T(\partial_r,\nabla u_i)^2\big\}\dm.\\
\nonumber&=\int_{\Omega}\big\{u_i^2(\mathscr{L}r)^2+2T(\partial_r,\Lu r\nabla u^2_i)+4\langle\partial_r,T(\nabla u_i)\rangle^2\big\}\dm.\\
&\leq \int_{\Omega}\Big\{u_i^2(\mathscr{L}r)^2+2T\Big(\partial_r,\nabla(u^2_i\Lu r)-u^2_i\nabla \Lu r\Big)\Big\}\dm+4\int_\Omega |T(\nabla u_i)|^2\dm.
\end{align}
The integration by parts formula gives us
\begin{equation}\label{Aux-Ju2}
\int_{\Omega}\big\{u_i^2(\mathscr{L}r)^2+2T(\partial_r,\nabla(u^2_i\Lu r))\big\}\dm = - \int_{\Omega}u_i^2(\mathscr{L}r)^2\dm.
\end{equation}
Moreover,
\begin{equation*}
\varepsilon\int_\Omega|\nabla u_i|^2\dm\leq \int_\Omega \g{\nabla u_i}{T(\nabla u_i)}\dm=\lambda_i.
\end{equation*}
So,
\begin{equation}\label{Aux-Ju3}
\int_\Omega |T(\nabla u_i)|^2\dm\leq \delta^2\int_\Omega|\nabla u_i|^2\dm\leq\frac{\delta^2\lambda_i}{\varepsilon}.
\end{equation}
Using \eqref{Aux-Ju1}, \eqref{Aux-Ju2} and \eqref{Aux-Ju3}, we get
\begin{eqnarray}\label{Eq2-Prop1GM}
&&\int_{\Omega}\big\{u_i^2(\mathscr{L}r)^2+4u_i\Lu rT(\partial_r,\nabla u_i)+4T(\partial_r,\nabla u_i)^2\big\}\dm.\nonumber\\
&\leq& - \int_{\Omega}u_i^2\big\{(\mathscr{L}r)^2+2T(\partial_r,\nabla \Lu r)\big\}\dm + 4\frac{\delta^2\lambda_i}{\varepsilon}.
\end{eqnarray}
Substituting \eqref{Eq2-Prop1GM} into \eqref{Eq1-Prop1GM}, we obtain the estimate of the proposition.
\end{proof}

We are now ready to prove our main results.

\subsection{Proof of Theorem~\ref{propCYT-c3}}
\begin{proof}
We begin by proving the result of the theorem for the hyperbolic space $\mathbb{H}^n(-1)$ case with constant curvature $-1$, i.e., the open half-space $x_n >0$ with its standard metric $g_{ij}=x_n^{-2}\delta_{ij}$. For this case, we have
\begin{equation*}
\nabla\ln x_n = g^{ij}\partial_i\ln x_n\partial_j = x_n\partial_n \quad \hbox{and} \quad |\nabla \ln x_n|=1.
\end{equation*}
By hypotheses $T(\nabla\ln x_n)=\psi\nabla\ln x_n=\psi x_n\partial_n$ and $\partial_n \eta=0$, thus 
\begin{eqnarray*}
\g{\nabla\eta}{T(\nabla\ln x_n)} = \g{g^{ij}\partial_i\eta\partial_j}{\psi x_n\partial_n} = 0.
\end{eqnarray*}
Recall that $\dv(a^i\partial_i)=\frac{1}{\sqrt{det(g_{ij})}}\partial_i(\sqrt{det(g_{ij})}a^i)$, and by assumption $\partial_n\psi=0$, so that
\begin{eqnarray*}
\mathscr{L}(\ln x_n)&=&\dv(T(\nabla \ln x_n))-\g{\nabla\eta}{T(\nabla\ln x_n)}\\
&=& \dv(\psi x_n\partial_n) = x_n^n\partial_n(\psi x_n^{1-n}) = (1-n)\psi.
\end{eqnarray*}
So,
\begin{equation*}
\langle T(\nabla\ln x_n),\nabla\mathscr{L}(\ln x_n)\rangle = \langle\psi x_n\partial_n,(1-n)\nabla\psi\rangle = (1-n)\psi x_n\partial_n\psi=0.
\end{equation*}
Since $\ln x_n$ works as a distance function on $\mathbb{H}^n(-1)$, we can take $r=\ln x_n$ in Proposition~\ref{Prop1GM} to obtain
\begin{eqnarray*}
\sum_{i=1}^k(\lambda_{k+1}-\lambda_i)^2\leq\frac{1}{\varepsilon}\sum_{i=1}^k(\lambda_{k+1} -\lambda_i)\Big(\frac{4\delta^2}{\varepsilon}\lambda_i - \int_{\Omega}u_i^2(1-n)^2\psi^2\dm\Big).
\end{eqnarray*}
Since $\varepsilon\leq \langle T(\nabla\ln x_n),\nabla\ln x_n\rangle=\psi$, we have $-\psi^2\leq-\varepsilon^2$, which is enough to obtain the next inequality 
\begin{equation*}
\sum_{i=1}^k(\lambda_{k+1}-\lambda_i)^2 \leq\frac{1}{\varepsilon}\sum_{i=1}^k(\lambda_{k+1} -\lambda_i)\Big(\frac{4\delta^2}{\varepsilon}\lambda_i-(n-1)^2\varepsilon^2\Big).
\end{equation*}
Now, by rescaling the metric by a factor of $\kappa^{-2}$, the previous inequality for the case of $\mathbb{H}^n(-\kappa^2)$ becomes
\begin{equation*}
\displaystyle\sum_{i=1}^k(\lambda_{k+1}-\lambda_i)^2 \leq\frac{1}{\varepsilon}\sum_{i=1}^k(\lambda_{k+1} -\lambda_i)\Big(\frac{4\delta^2}{\varepsilon}\lambda_i-(n-1)^2\varepsilon^2\kappa^2\Big).
\end{equation*}
Moreover, quadratic estimate \eqref{Eq1-Prop1GM} guarantees that 
\begin{equation*}
\frac{4\delta^2}{\varepsilon}\lambda_i-(n-1)^2\varepsilon^2\kappa^2\geq0,
\end{equation*}
for $i\in\{1,2,\ldots\},$ but the interesting case is $\lambda_1\geq\frac{\varepsilon}{4\delta^2}(n-1)^2\varepsilon^2\kappa^2>0,$  since $n\geq2$ and $\kappa>0.$
\end{proof}

\subsection{Proof of Theorem~\ref{thmA3-c3}} 
\begin{proof}
By Proposition~\ref{Prop1GM}, we just need to estimate the integral
\begin{equation*}
-\int_{\Omega}u_i^2\big((\mathscr{L}r)^2+2\g{T\partial_r}{\nabla \Lu r}\big)\dm,
\end{equation*}
for the particular case of the drifted Cheng--Yau $\Lu r=\Box r-\g{T\partial_r}{\nabla \eta}$. First, we compute
\begin{eqnarray*}
&&(\mathscr{L}r)^2+2\g{T\partial_r}{\nabla \Lu r}\\
&&=(\Box r)^2-2\Box r \g{T\partial_r}{\nabla\eta}+\g{T\partial_r}{\nabla\eta}^2+2\g{T\partial_r}{\nabla \Box r}-2\g{T\partial_r}{\nabla \g{T\partial_r}{\nabla\eta}}\\
&&=(\Box r)^2+2\g{T\partial_r}{\nabla \Box r}-2\Box r \g{T\partial_r}{\nabla\eta}-2\g{T\partial_r}{\nabla \g{T\partial_r}{\nabla \eta}}+\g{T\partial_r}{\nabla\eta}^2.
\end{eqnarray*}
Therefore
\begin{align*}
-&\int_{\Omega}u^2_i\big((\mathscr{L}r)^2+2\g{T\partial_r}{\nabla \Lu r}\big)\dm\\
=&\int_{\Omega}u^2_i\Big(-(\Box r)^2-2\g{T\partial_r}{\nabla \Box r}\Big)\dm 
+2\int_\Omega u^2_i\Box r \g{T\partial_r}{\nabla\eta}\dm\\
&+\int_\Omega u^2_i \Big(2\g{T\partial_r}{\nabla \g{T\partial_r}{\nabla \eta}}-\g{T\partial_r}{\nabla\eta}^2\Big)\dm.
\end{align*}
Since $\nabla_{\partial_r}T=0$ and $T\partial_r=\psi\partial_r$, we have that $\partial_r\psi=0$ and 
\begin{equation*}
\g{\partial_r}{\nabla \g{T\partial_r}{\nabla \eta}}=\partial_r\g{T\partial_r}{\nabla \eta}=\psi\ddot{\eta}.
\end{equation*}
Thus,
\begin{equation*}
\int_\Omega u^2_i \Big(2\g{T\partial_r}{\nabla \g{T\partial_r}{\nabla \eta}}-\g{T\partial_r}{\nabla\eta}^2\Big)\dm \leq \int_\Omega u^2_i \psi^2 \Big(2\ddot{\eta}-\dot{\eta}^2\Big)\dm\leq C_1,
\end{equation*}
where $C_1=\delta^2 \max_{\bar{\Omega}}(2\ddot{\eta}-\dot{\eta}^2)$. The previous estimate together with Propositions~\ref{lemmaboxestimate} and \ref{Prop-C0} immediately implies the inequalities of Theorem~\ref{thmA3-c3}.
\end{proof}

\section{Applications of Theorem~\ref{propCYT-c3}}\label{Sec5}
As in the other proofs in this paper, we proceed in stages. First, we recall a result of McKean~\cite{McKean}; see also Chavel~\cite{Chavel}. We then provide a proof by combining Theorem~\ref{propCYT-c3} with a spectral approach based on warped product metrics.

\begin{lemma}\label{Ball-lemma}
Let $B(a)\subset \mathbb{H}^n(-\kappa^2)$ be an $n$-disk of radius $a>0$. The first eigenvalue $\lambda_1$ of the Laplacian on $B(a)$ with the Dirichlet boundary condition satisfies
\begin{eqnarray*}
\lim_{a\to +\infty}\lambda_1(B(a))=\frac{(n-1)^2}{4}\kappa^2.
\end{eqnarray*}
\end{lemma}
\begin{proof}
We start by proving the result of the lemma for the case $\mathbb{H}^n(-1)$. Recall that its metric in the polar coordinates is given by $\langle,\rangle=dr^2+\sinh^2rds_{n-1}^2$ on $(0,+\infty)\times S^{n-1}$. Thus, we can address the eigenvalue problem as in \cite[Section~2]{Marrocos-Gomes}, from which we take $\mu_0=0$ and $\psi$ to be constant, so that our eigenvalue problem for the Laplacian becomes
\begin{equation*}
\ddot{\phi}(r)+(n-1)\coth(r)\dot{\phi}(r)+\lambda\phi(r)=0,
\end{equation*}
for some $\phi\in L^2\Big([\frac{a}{2},a],\sinh^{2(n-1)}(r)dr^2\Big)$. For our purpose, it is enough to consider $\lambda=\frac{(n-1)^2}{4}$ and the function $\coth(r)$ for large values of $r$. So, from now on, we are considering the solution of the ODE on $\big(\frac{a}{2},a\big)$ as follows
\begin{equation*}
\ddot{\phi}(r)+(n-1)\dot{\phi}(r)+\frac{(n-1)^2}{4}\phi(r)=0.
\end{equation*}
We define $f$ on $B(a)\subset\mathbb{H}^n(-1)$ given by $f(r,\theta)=\psi(r)$, where
\begin{equation*}
\psi(r)=\left\{
\begin{array}{ccc}
\phi(r)&\hbox{if} & r\in(\frac{a}{2},a),\\
0&& \hbox{otherwise}.
\end{array}
\right.
\end{equation*}
Then $f$ is an admissible function for the Dirichlet eigenvalue problem on the space $L^2(B(a),d\nu)$, where $d\nu=\sinh^{n-1}(r)dvol_{B(a)}$. Besides, $\psi$ satisfies
\begin{equation*}
\ddot{\psi}(r)+(n-1)\dot{\psi}(r)+\frac{(n-1)^2}{4}\psi=0 \quad\hbox{on}\quad \Big(\frac{a}{2},a\Big).
\end{equation*}
Now, we use integration by parts to obtain
\begin{eqnarray*}
\int_{B(a)}\dot{\psi}^2d\nu
= \int_{B(a)}\Big((n-1)\psi\dot{\psi}(1-\coth(r))+\frac{(n-1)^2}{4}\psi^2\Big)d\nu
\end{eqnarray*}
which implies  
\begin{equation*}
\int_{B(a)}\Big(\dot{\psi}^2 -\frac{(n-1)^2}{4}\psi^2\Big)d\nu
\leq \int_{B(a)}(n-1)|\psi| |\dot{\psi}||\coth(r)-1|d\nu
\end{equation*}
Note that, $\sup_{B(a)}|\coth(r)-1|=|\coth{(a/2)}-1|$  and $|\dot{\psi}|= |\nabla f|$ on $B(a)$. So, 
\begin{eqnarray*}
\Vert\nabla f\Vert^2 -\frac{(n-1)^2}{4}\Vert f\Vert^2
\leq (n-1)|\coth(a/2)-1| \Vert f\Vert \Vert \nabla f\Vert.
\end{eqnarray*}
The previous inequality reads as 
\begin{eqnarray*}
\frac{\Vert \nabla f\Vert^2}{\Vert f\Vert^2 }-(n-1)|\coth(a/2)-1|\frac{\Vert \nabla f\Vert}{\Vert  f\Vert} -\frac{(n-1)^2}{4}\leq 0.
\end{eqnarray*}
Hence, we get 
\begin{eqnarray*}
\frac{\Vert \nabla f\Vert}{\Vert f\Vert }\leq \frac{n-1}{2}|\coth(a/2)-1|+ \frac{1}{2}\Big((n-1)^2|\coth(a/2)-1|^2 +(n-1)^2 \Big)^\frac{1}{2}.
\end{eqnarray*}
By Rayleigh's theorem (see, e.g., Chavel~\cite{Chavel}) 
\begin{equation*}
\sqrt{\lambda(B(a))}\leq\frac{n-1}{2}|\coth(a/2)-1|+ \frac{1}{2}\Big((n-1)^2|\coth(a/2)-1|^2 +(n-1)^2 \Big)^\frac{1}{2}.
\end{equation*}
Thus, $\lim_{a\to+\infty}\lambda_1(B(a))\leq\frac{(n-1)^2}{4}.$ From this latter result and Theorem~\ref{propCYT-c3}, we get  $\lim_{a\to+\infty}\lambda_1(B(a))=\frac{(n-1)^2}{4}$. We now use the same argument as in the proof of Theorem~\ref{propCYT-c3} to conclude that $\lim_{a\to+\infty}\lambda_1(B(a))=\frac{(n-1)^2}{4}\kappa^2$ for the case of $\mathbb{H}^n(-\kappa^2).$
\end{proof}

\subsection{Proof of Corollary~\ref{newAproach}}
\begin{proof}
Here we are following the same steps of the proof given by Cheng-Yang~\cite[Corollary 1.3]{Cheng-Yang-III}. Let us consider into each bounded domain $\Omega\subset\mathbb{H}^n(-\kappa^2)$ an $n$-disk $B(a)$ of radius $a>0$. Thus, from the domain monotonicity of eigenvalues (see, e.g., Chavel~\cite{Chavel}), Lemma~\ref{Ball-lemma} and Theorem~\ref{propCYT-c3}, we have
\begin{eqnarray}\label{AuxEq1proofnewAproach}
\lambda_1(\Omega)\geq \frac{(n-1)^2}{4}\kappa^2\quad \hbox{and}\quad
\lim_{\Omega\to \mathbb{H}^n(-\kappa^2)}\lambda_1(\Omega)=\frac{(n-1)^2}{4}\kappa^2.
\end{eqnarray}
Moreover, it is clear that 
\begin{equation*}
\lambda_i(\Omega)>\lambda_1(\Omega)\geq \frac{(n-1)^2}{4}\kappa^2\quad \forall i>1.
\end{equation*}
Note that for the Laplacian case, we can work with $\upsilon_i:=4\lambda_i -(n-1)^2\kappa^2$, see Theorems~\ref{propCYT-c3} and~\ref{StandarEstimate-CYO}. Since $\upsilon_1=4\lambda_1 -(n-1)^2\kappa^2$, then from \eqref{AuxEq1proofnewAproach} we obtain $\upsilon_1(\Omega)\to 0$ as $\Omega \to \mathbb{H}^n(-\kappa^2)$. 
On the other hand, from \eqref{2-Aux-Num}, we get $\upsilon_{i+1}\leq 5i^2\upsilon_1$ $\forall i\geq1$. Consequently,
\begin{equation*}
0=\lim_{\Omega\to \mathbb{H}^n(-\kappa^2)}\upsilon_{i+1}(\Omega)=\lim_{\Omega\to \mathbb{H}^n(-\kappa^2)}\big(4\lambda_{i+1}(\Omega) -(n-1)^2\kappa^2\big)
\end{equation*}
which is enough to complete the proof of the corollary.
\end{proof}

\section{The fundamental gap for a class of operators on a class of convex domains in two-dimensional hyperbolic space}\label{secaogap}

In this section, we address the fundamental gap for a certain class of operators $\Lu$ defined as in~\eqref{defLintro}. More precisely, we estimate the difference $\lambda_2(\Omega)-\lambda_1(\Omega)$ between the first two eigenvalues of the eigenvalue problem for a special family of $\Lu$ on a class of convex bounded domains $\Omega$ with the Dirichlet boundary condition. In the first part, we are working with the class of operators $\Lu = \dv(\varphi\nabla u),$ for some radially constant function $\varphi$, satisfying $\varepsilon\leq\varphi\leq\delta,$ for some positive constants $\varepsilon$ and $\delta.$ 

We begin with a brief historical background. The solution for the problem $y''(t)+\lambda y(t)=0$ in $ (0,\ell),$ with $\lambda>0$ and $y(0)=y(\ell)=0$, is given by $y(t)=\sum_n\big(A_n\sin(\sqrt{\lambda_n}t)+B_n\cos({\sqrt{\lambda_n}t})\big),$ where $\lambda_n=(n\pi)^2/\ell^2 $, hence, $\lambda_2-\lambda_1=3\pi^2/\ell^2.$ This gap motivates us to think about the more general case of the Laplacian in convex bounded domains $\Omega\subset\mathbb{R}^n$. For this case, it was observed in the 80s by Michiel van den  Berg~\cite{MvdB} that for many convex domains, $\lambda_2-\lambda_1\geq3\pi^2/D^2,$ where $D$ is the diameter of $\Omega.$ It was also independently suggested by Ashbaugh and Benguria~\cite{AB}, and Yau~\cite{Y1} that this estimate of the gap remains true for any convex bounded domain in $\mathbb{R}^n$. It has been known as the \emph{fundamental gap conjecture}. For the case of non-convex domains, it is known that the fundamental gap has no such a lower bound, and for non-connected cases, the gap may vanish.

In 2011, Andrews and Clutterbuck~\cite{Andrews} proved this conjecture and suggested that it remains valid for the case of constant curvature spaces. In 2019, Seto, Wang and Wei~\cite{Seto} proved this new conjecture for convex domains with diameter $D\leq \pi/2$ in the unit sphere $\mathbb{S}^n$, $n\geq3$. In 2020, He, Wei and Zhang~\cite{He} extended Seto-Wang-Wei's result to convex domains with diameter $D<\pi$ in the unit sphere $\mathbb{S}^n$, $n\geq3$. In 2021, Dai, Seto and Wei~\cite{Dai} proved the conjecture for any convex bounded domains in $\mathbb{S}^2$. 

However, the fundamental gap for the Laplacian behaves differently in negatively curved spaces. Indeed, Bourni et al.~\cite{Bourni} constructed convex domains in $\mathbb{H}^2(-1)$, with diameter $D$, such that $(\lambda_2-\lambda_1)D^2< 3\pi^2$, and as the quantity $(\lambda_2-\lambda_1)D^2$ is invariant under the scaling of the metric, this same result also holds for any simply connected negative constant curvature space forms. Very recently, for the more general case of convex domains in $\mathbb{H}^n(-1)$, $n\geq2$, the same authors proved that $(\lambda_2-\lambda_1)D^2$ can be arbitrarily small for domains of any diameter (see Bourni et al.~\cite[Theorem~1.1]{Bourni1}). 

The essence of these results motivated us to work on the behavior of the fundamental gap for some cases of our operator $\Lu$ on convex bounded domains in hyperbolic space. Here, we find a simple class of radially parallel $(1,1)-$tensors $T$ and drifting functions $\eta$ that define the operator $\Lu$ in \eqref{defLintro}, to answer positively the following two questions:

\emph{Is there some convex domain in hyperbolic space for which the fundamental gap for the operator $\dv(\varphi\nabla u)$ satisfies: $(\lambda_2-\lambda_1)D^2< 3\pi^2\delta$?}

\emph{ Is there some drifting function $\eta$ for which the fundamental gap for the operator $\dv(\varphi\nabla u)-\g{\nabla \eta}{\varphi\nabla u}$ still satisfies the inequality $(\lambda_2-\lambda_1)D^2< 3\pi^2\delta$ on some convex domain in hyperbolic space?}

Bourni et al.~\cite{Bourni} provide the domains and the separation-of-variables method used below. We adapt their method to the setting considered here and establish some generalizations of their results; the Laplacian case is due to~\cite{Bourni}.

Due to the invariance of the quantity $(\lambda_2-\lambda_1)D^2$, we can work, without loss of generality, on the hyperbolic space $\mathbb{H}^2(-1)$ with constant curvature~$-1$, namely, the open half-space $y>0$ with the metric $g_{ij}=y^{-2}\delta_{ij}$. Such a metric in the coordinates $x=r\cos\theta$ and $y=r\sin\theta$, where $r>0$ and $0<\theta<\pi,$ is written as follows
\begin{equation}\label{metrica}
ds^2 =\frac{1}{\sin^2\theta}\frac{dr^2}{r^2}+\frac{d\theta^2}{\sin^2\theta}
\end{equation}
so that $\{e_1=r\sin\theta \frac{\partial}{\partial r},\, e_2=\sin\theta \frac{\partial}{\partial \theta}\}$ is an orthonormal frame, whose nonzero Christoffel symbols are $\Gamma^2_{11}= -\Gamma^1_{12}=-\Gamma^1_{21}= \cos\theta$.  Note that we can define an $(1,1)-$tensor $T$ on $\mathbb{H}^2(-1)$ by $T=\varphi I$, for $\varphi\in C^{\infty}(\mathbb{H}^2)$, with $\varepsilon\leq\varphi\leq\delta$, to be appropriately chosen.

First we compute $\dv(\varphi\nabla u)$ of a smooth function $u$ on $\mathbb{H}^2(-1)$. For this, we set $\nabla u=u_1e_1+u_2e_2$, where $u_1=e_1(u)=r\sin \theta u_r$ and $u_2=e_2(u)=\sin\theta u_\theta$. Hence,
\begin{eqnarray*}
\g{\nabla_{e_1}\nabla u}{e_1}&=&(u_{11}+u_2\Gamma^1_{12})\,=\, r^2\sin^2\theta u_{rr}+r\sin^2\theta u_r-\sin\theta\cos\theta u_\theta,\\
\g{\nabla_{e_2}\nabla u}{e_2}&=&u_{22}\,=\,\sin^2\theta u_{\theta\theta}+\sin\theta\cos\theta u_\theta.
\end{eqnarray*}
Thus,
\begin{eqnarray*}
\nonumber\Delta u&=&\g{\nabla_{e_1}\nabla u}{e_1}+\g{\nabla_{e_2}\nabla u}{e_2}\\
\nonumber&=&r^2\sin^2\theta u_{rr}+r\sin^2\theta u_r-\sin\theta\cos\theta u_\theta+\sin^2\theta u_{\theta\theta}+\sin\theta\cos\theta u_\theta\\
&=&r^2\sin^2\theta u_{rr}+r\sin^2\theta u_r+\sin^2\theta u_{\theta\theta}.
\end{eqnarray*}
Besides, $\varphi_1=e_1(\varphi)=r\sin \theta \varphi_r$ and $\varphi_2=e_2(\varphi)=\sin\theta \varphi_\theta,$ thus
\begin{equation}\label{PI-quadrado}
\g{\nabla u}{\nabla \varphi}=u_1\varphi_1+u_2\varphi_2=r^2\sin^2\theta \varphi_ru_r+\sin^2\theta \varphi_\theta u_\theta,
\end{equation}
and recall that
\begin{equation}\label{quadrado T geral}
    \dv(\varphi\nabla u)=\varphi\Delta u+\g{\nabla \varphi}{\nabla u}.
\end{equation}
Under the additional assumption of $T$ be radially parallel, we get $\varphi_r=0$, whence equation~\eqref{quadrado T geral} becomes
\begin{equation}\label{opchengyau}
\dv(\varphi\nabla u)=\varphi(r^2\sin^2\theta u_{rr}+r\sin^2\theta u_r)+\varphi\sin^2\theta u_{\theta\theta}+\sin^2\theta \varphi_\theta u_\theta,
\end{equation}
which completes the first part.

Now, we are working on the fundamental gap for $\dv(\varphi\nabla u)$ in \eqref{opchengyau}. We start with the following eigenvalue problem with the Dirichlet boundary condition:
\begin{align}\label{eq4chengyau}
\varphi(r^2\sin^2\theta u_{rr}+r\sin^2\theta u_r)+\varphi\sin^2\theta u_{\theta\theta}+\sin^2\theta\varphi_\theta u_\theta+\lambda u&=0, \quad \mbox{in} \quad \Omega,\\
\nonumber u&=0, \quad \mbox{on}\quad \partial \Omega.
\end{align}

For each $\ell > 0$, $\theta_0\in (0,\frac{\pi}{2})$ and $\theta_1\in (\frac{\pi}{2},\pi)$ consider the family of domains $\Omega_{\ell,\theta_0,\theta_1}=\{(r,\theta)\,:\,1<r<e^\frac{\pi}{\ell}\,\,\mbox{and}\,\, \theta_0<\theta <\theta_1\}$, see Figure~\ref{fig1}. As the geodesics are the vertical lines $x = \ell $ and semicircles centered on the axis $x$, then the sets $ \Omega_{\ell,\theta_0,\theta_1}$ are convex domains in $\mathbb{H}^2(-1).$
\begin{figure}[!htb]
\centering
\includegraphics[scale=0.4]{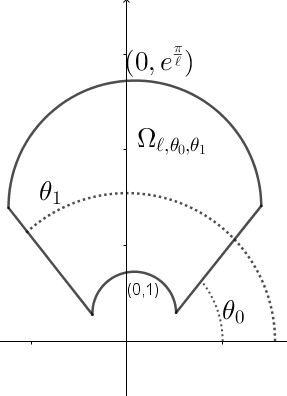}
\caption{$\Omega_{\ell,\theta_0,\theta_1}
= \{(r,\theta) : 1 < r < e^{\pi/\ell},
\ \theta_0 < \theta < \theta_1\}$. 
Adapted from Bourni et al.~\cite{Bourni}.}\label{fig1}
\end{figure}
Since the metric in~\eqref{metrica} is a warped product, we can use the method of separating variables (see, e.g., \cite[page~41]{Chavel}), from which we write
$u(r,\theta)=f(r)h(\theta),$ so that $u_{r}=f_{r}h$, $u_{rr}=f_{rr}h$, $u_{\theta}=fh_{\theta}$ and $u_{\theta\theta}=fh_{\theta\theta}$. Hence, from Problem~\eqref{eq4chengyau}, we get 
\begin{equation}
(r^2f_{rr} +rf_r)h+(h_{\theta\theta}+\frac{\varphi_\theta}{\varphi}h_\theta +\frac{\lambda}{\varphi} \csc^2\theta h)f=0.
\end{equation}
Since $f$ depends only on $r$, and $h$ depends only on $\theta$, there exists a constant $\mu$ such that we can replace this problem by the following two eigenvalue problems
\begin{eqnarray}
\label{eq1chengyau} r^2f_{rr}+rf_r=-\mu f, \quad r\in (1,e^\frac{\pi}{\ell})\\
\label{eq2chengyau} h_{\theta\theta}+\frac{\varphi_\theta}{\varphi}h_\theta +\frac{\lambda}{\varphi} \csc^2\theta h=\mu h, \quad \theta\in (\theta_0,\theta_1),
\end{eqnarray}
with the Dirichlet bounded conditions $f(1)=f(e^\frac{\pi}{\ell})=0$ and $h(\theta_0)=h(\theta_1)=0$, respectively. Besides, since $1<r<e^\frac{\pi}{\ell}$, we can make
the change of variable $f(r)=f(e^t),$ for $0< t <\frac{\pi}{\ell}.$ Whence, $f_t=rf_r$ and $f_{tt}= rf_r+ r^2f_{rr}$. So, from~\eqref{eq1chengyau}, we obtain
\begin{equation}\label{eq3chengyau}
f_{tt}=-\mu f, \quad  t\in \big(0,\frac{\pi}{\ell}\big),
\end{equation}
and, thus, the Dirichlet bounded condition guarantees that $f (t) = \sin (\sqrt{\mu} t)$, with $\mu = (k\ell)^2>0$, where $k$ is a nonzero integer, i.e., $f (t) = \sin (k\ell t)$. While Eq.~\eqref{eq2chengyau} is rewritten as  
\begin{equation}\label{eq2chengyau-Equiv}
-(\varphi h_\theta)_\theta+\mu \varphi h=\lambda \csc^2\theta h. 
\end{equation}

\subsection{Identifying the first two eigenvalues}\label{Id.autovalores}
For the sake of completeness, we start this section by stating the Courant results for nodal domains, the Sturm comparison theorem for Jacobi equations, and the Sturm--Liouville theorem.

\begin{theorem}[Courant results for nodal domains, see, e.g., ~\cite{AHenrot}, p.~14-15]\label{teorsinalu_1} 
The first eigenfunction $u_1$ of a second-order elliptic differential operator with Dirichlet boundary condition is positive in $\Omega$ when $\Omega\subset\mathbb{R}^n$ is connected; the first eigenvalue $\lambda_1$ has multiplicity equal to 1; and the second eigenfunction $u_2$ has precisely $2$ nodal domains. Moreover, $\lambda_1 $ is characterized as being the only eigenvalue with an eigenfunction of constant sign.
\end{theorem}

\begin{theorem}[Sturm comparison theorem, see~\cite{Jorge}, p.~104] \label{teosturn}
Let $f_i$ be non-trivial solutions of
\begin{equation*}
(p(x)f_i'(x))'+b_i(x)f_i(x) = 0, \quad \mbox{for}\quad i=1,2,
\end{equation*}
where $0<p\in C^1$, $b_1$ and $b_2$ are continuous functions, and $b_1\geq b_2$ for all $x$. If $x_1<x_2$ are two consecutive zeros of $f_2$, then $f_1$ has at least one zero in  $(x_1,x_2)$, unless $b_1(x)=b_2(x)$ and $f_2(x)=kf_1(x)$, $k \in \mathbb{R}$.
\end{theorem}

We now highlight the well-known eigenvalue problem of Sturm--Liouville.
\begin{eqnarray}\label{problssturmliouv}
\nonumber -(py')'+qy&=&\lambda \rho y\,\,\mbox{on}\,\, [a,b],\\
y(a)=y(b)&=&0,
\end{eqnarray}
with the coefficients $p$ and $q$, and the weight function $\rho$ satisfying:
\begin{eqnarray}\label{condtionsturmlioupro}
\nonumber p\in C^1[a,b],\,\,q,\,\,\rho\in C[a,b],\\
p(x)\geq d>0 \,\,\mbox{and}\,\, q(x)+c_2\rho(x)\geq 0\,\,\mbox{for all}\,\,x\in[a,b],\\
\nonumber\rho(x)>0\,\, \mbox{for all}\,\, x\in (a,b).
\end{eqnarray}
\begin{theorem}[Sturm--Liouville, see, e.g.,~\cite{Kielhofer}, p.~174-175] \label{teosturm1} Under the hypotheses~\eqref{condtionsturmlioupro}, the Sturm--Liouville eigenvalue problem~\eqref{problssturmliouv} possesses infinitely many linear independent eigenfunctions $u_n\in C^2[a,b]$ with eigenvalues $\lambda_n\in\mathbb{R}$, which satisfy 
\begin{eqnarray*}
\int_a^b\rho u_nu_m=\delta_{nm} \quad \mbox{and}\quad\lambda_1<\lambda_2<\lambda_3<\cdots \to \infty.
\end{eqnarray*}
Moreover, each of the eigenvalue $\lambda_n$ not only have geometric multiplicity one, but by symmetry also algebraic multiplicity one, and the $n^{th}$ eigenfunction $u_n$ of a Sturm--Liouville eigenvalue problem~\eqref{problssturmliouv} has at most $n-1$ simple zeros in $(a,b).$
\end{theorem}

According to Theorem~\ref{teorsinalu_1}, the first eigenvalue $\lambda_1$ of Problem~\eqref{eq4chengyau} on $\Omega_{\ell,\theta_0,\theta_1}$ is a strictly positive eigenfunction. Thus, $\mu=\ell^2$, because $f>0$ on $(0,\frac{\pi}{\ell})$, besides, $h>0$ on $(\theta_0,\theta_1)$, and from Theorem~\ref{teosturm1}, we have that $\lambda_1$ is the smallest $\lambda^{\ell^2}_1$ that solves the following problem:
\begin{eqnarray}\label{eq.8}
-(\varphi h_{\theta})_\theta +\ell^2\varphi h &=& \lambda\csc^2\theta h,\quad  \theta \in (\theta_0,\theta_1),\\
\nonumber h(\theta_0)=h(\theta_1)&=&0.
\end{eqnarray}
Again, from Theorem~\ref{teorsinalu_1}, we know that $\lambda_2$ is an eigenfunction that changes sign only once, so $f$ or $h$ has to change sign. 
 
If $f$ change of sign, then $\mu = 4\ell^2$ and $f(t)=\sin(2\ell t)$ in Problem~\eqref{eq3chengyau}. In this case, $h>0$, and thus, Theorem~\ref{teosturm1} guarantees that $\lambda_2$ is the smallest $\lambda^{4\ell^2}_1$ that solves the problem
\begin{eqnarray}\label{eq.9}
-(\varphi h_{\theta})_\theta + 4\ell^2\varphi &=& \lambda\csc^2\theta h, \quad \theta\in (\theta_0,\theta_1)\\
\nonumber h(\theta_0)=h(\theta_1)&=&0.
\end{eqnarray}

If $h$ change of sign, then $f$ is positive, and given by $\sin(\ell t)$, with $\mu = \ell^2$. In this case, $\lambda_2$ is given by $\lambda^{\ell^2}_2$ solving~\eqref{eq.8} with $h$ changing of sign exactly once, because Theorem~\ref{teosturm1}. Hence, the second eigenvalue is:
\begin{equation}\label{segundoautovalor}
    \lambda_2 = \min\{\lambda^{4\ell^2}_1, \lambda^{\ell^2}_2\}.
\end{equation}

\subsection{Estimates on the first and second eigenvalues}\label{sec3}
In this section, we will compute estimates for the first two eigenvalues for the operator in~\eqref{opchengyau} with the Dirichlet boundary condition on a domain to be appropriately chosen. We begin by defining the angle:
\begin{equation*}
\theta_\divideontimes=\min(\theta_0,\pi-\theta_1).
\end{equation*}
Note that $\csc^2\theta$ is a decreasing function on $(0,\pi/2]$ and increasing on $[\pi/2,\pi)$, so that  $1\leq \csc^2\theta \leq \csc^2(\theta_\divideontimes)$, for all $\theta \in [\theta_0,\theta_1]$. 
\begin{lemma}\label{lema1}
The first eigenvalue $\lambda^\mu_1$ of Problem~\eqref{eq2chengyau} satisfies
\begin{equation}\label{eq.10}
\varepsilon \sin^2(\theta_\divideontimes)\Big(\mu + \frac{\pi^2}{(\theta_1-\theta_0)^2}\Big)\leq \lambda^\mu_1\leq \delta\Big(\mu + \frac{\pi^2}{(\theta_1-\theta_0)^2}\Big).
\end{equation}
\end{lemma}

\begin{proof}
For the lower estimate, we consider a solution $h$ of~\eqref{eq2chengyau}. We multiply by $h$ both sides of \eqref{eq2chengyau-Equiv}, and we integrate from $\theta_0$ to $\theta_1$, to get
\begin{equation*}
\lambda_1^{\mu}= \frac{\int_{\theta_0}^{\theta_1}\varphi((h_\theta)^2+\mu h^2)d\theta}{\int_{\theta_0}^{\theta_1}\csc^2 \theta h^2d\theta}
\geq \frac{\varepsilon}{\csc^2(\theta_\divideontimes)}\Big( \mu + \frac{\int_{\theta_0}^{\theta_1}(h_\theta)^2d\theta}{\int_{\theta_0}^{\theta_1}h^2d\theta}\Big).
\end{equation*}
By using the Wirtinger inequality $\int^D_0(h')^2dx \geq \frac{\pi^2}{D^2}\int_0^Dh^2dx,$ we obtain
\begin{equation*}
\lambda_1^{\mu} \geq  \varepsilon\sin^2(\theta_\divideontimes)\Big(\mu +\frac{\pi^2}{(\theta_1-\theta_0)^2}\Big).
\end{equation*}

For the upper estimate, we choose the test function $\phi=\sin(\frac{\theta-\theta_0}{\theta_1-\theta_0}\pi)$. By Rayleigh's theorem (see, e.g.,\cite[p.~177]{Kielhofer}) and the fact that $\csc^2\theta\geq 1$, we have
\begin{align*}
\lambda_1^{\mu}\leq  \frac{\int_{\theta_0}^{\theta_1}\varphi((\phi_\theta)^2+\mu \phi^2)d\theta}{\int_{\theta_0}^{\theta_1}\csc^2 \theta \phi^2d\theta}
\leq \delta\Big( \mu + \frac{\int_{\theta_0}^{\theta_1}(\phi_\theta)^2d\theta}{\int_{\theta_0}^{\theta_1}\phi^2d\theta}\Big)
=  \delta\Big(\mu +\frac{\pi^2}{(\theta_1-\theta_0)^2}\Big).
\end{align*}
\end{proof}

\begin{lemma}\label{lema2}
The second eigenvalue $\lambda^\mu_2$ of Problem~\eqref{eq2chengyau} satisfies
\begin{equation}\label{eq.11}
\varepsilon\sin^2\theta_\divideontimes \Big(\mu+\frac{4\pi^2}{(\theta_1-\theta_0)^2}\Big)\leq \lambda^\mu_2\leq \delta\Big(\mu+\frac{4\pi^2}{(\theta_1-\theta_0)^2}\Big).
\end{equation}
\end{lemma}
\begin{proof}
Let $h^\mu_2$ be an eigenfunction corresponding to the second eigenvalue $\lambda^\mu_2$. From Theorem~\ref{teosturm1}, there is only one $\theta_2\in (\theta_0,\theta_1)$ such that $h^\mu_2(\theta_2) = 0$, besides, the eigenvalue $\lambda^\mu_2$ coincides with the first eigenvalue of $-(\varphi h_{\theta})_\theta +\mu \varphi h=\lambda\csc^2\theta h$ with the Dirichlet bounded condition on any of the intervals $[\theta_0,\theta_2]$ or $[\theta_2,\theta_1]$. The lower and upper limits on~\eqref{eq.11} are obtained by Lemma~\ref{lema1} considering the intervals with the shortest and longest lengths among $[\theta_0,\theta_2]$ and $[\theta_2,\theta_1]$, respectively, and noticing that $\min\{\theta_2-\theta_0,\theta_1-\theta_2\}\leq (\theta_1-\theta_0)/2$ and $\max\{\theta_2-\theta_0,\theta_1-\theta_2\}\geq (\theta_1-\theta_0)/2.$
\end{proof}

Now, $\lambda_2$ will be appropriately chosen. For that, we fix $\theta_\divideontimes> \frac{\pi}{6}$ and $\varepsilon> \delta/4$, so that $\frac{4\varepsilon\sin^2 \theta_\divideontimes-\delta}{4\delta-\varepsilon\sin^2 \theta_\divideontimes}>0$, thus we can choose $\ell$ such that 
\begin{equation}\label{eq.12}
\frac{\pi^2}{(\theta_1-\theta_0)^2}\frac{4\varepsilon\sin^2 \theta_\divideontimes-\delta}{4\delta-\varepsilon\sin^2 \theta_\divideontimes}\geq \ell^2.
\end{equation}
Taking $\mu=4\ell^2$ in~\eqref{eq.10} and $\mu=\ell^2$ in~\eqref{eq.11}, we get
\begin{eqnarray*}
\lambda_2^{\ell^2}-\lambda_1^{4\ell^2}&\geq& \varepsilon\sin^2\theta_\divideontimes \Big(\ell^2+\frac{4\pi^2}{(\theta_1-\theta_0)^2}\Big)-\delta\Big(4\ell^2 + \frac{\pi^2}{(\theta_1-\theta_0)^2}\Big)\\
&=& \ell^2\Big(\varepsilon\sin^2\theta_\divideontimes-4\delta\Big)+\frac{\pi^2}{(\theta_1-\theta_0)^2}\Big(4\varepsilon\sin^2\theta_\divideontimes-\delta \Big)\\
&=& \ell^2\Big(\varepsilon\sin^2\theta_\divideontimes-4\delta\Big)+\frac{\pi^2}{(\theta_1-\theta_0)^2}\Big(\frac{4\varepsilon\sin^2\theta_\divideontimes-\delta }{4\delta-\varepsilon\sin^2\theta_\divideontimes}\Big)(4\delta-\varepsilon\sin^2\theta_\divideontimes).
\end{eqnarray*}
Thus, by using~\eqref{eq.12}, we conclude that
\begin{eqnarray*}
\lambda_2^{\ell^2}-\lambda_1^{4\ell^2}\geq \ell^2\Big(\varepsilon\sin^2\theta_\divideontimes-4\delta\Big)+\ell^2\Big(4\delta-\varepsilon\sin^2\theta_\divideontimes\Big)=0.
\end{eqnarray*}
Hence, the second eigenvalue of Problem~\eqref{eq4chengyau} on $\Omega_{\ell,\theta_0,\theta_1}$ must be $\lambda^{4\ell^2}_1$, see~\eqref{segundoautovalor}. Geometrically, this corresponds to a domain as in Figure~\ref{fig1} in which the opening angle is small compared to the vertical length.

From now on, the operator $\dv(\varphi\nabla u))$ in \eqref{opchengyau} will be considered on the family of domains
\begin{equation}\label{OMEGA}
\Omega_{\ell,\theta_0,\theta_1}=\big\{(r,\theta)\,:\,1<r<e^\frac{\pi}{\ell}\quad\mbox{and}\quad \theta_0<\theta <\theta_1\,\, \mbox{satisfying}\,\,\eqref{eq.12}\big\}.
\end{equation}

To estimate the fundamental gap for $\dv(\varphi\nabla u)$, with $\varepsilon\leq\varphi\leq\delta$ and $\varphi_r=0$, we will need the diameter estimate of $\Omega_{\ell,\theta_0,\theta_1}$, which has been calculated by Bourni et al. in the more general configuration of this domain, namely:
\begin{lemma}\label{lema4.1}(Bourni et al.~\cite{Bourni})
 $\frac{\pi^2}{\ell^2D^2_{\ell,\theta_0,\theta_1}} \to 1$ as $\ell \to 0$ or $\theta_\divideontimes\to \frac{\pi}{2}$.
\end{lemma}

\subsection{Estimate of the fundamental gap}

\begin{lemma}\label{lema3} 
The fundamental gap for $\dv(\varphi\nabla u)$, with $\varepsilon\leq\varphi\leq\delta$ and $\varphi_r=0$, on each domain $\Omega_{\ell,\theta_0,\theta_1}$ defined by~\eqref{OMEGA}, satisfies 
\begin{equation}\label{eq.13}
3\varepsilon\sin^2\theta_\divideontimes \ell^2< \lambda_2-\lambda_1<3\delta \ell^2.
\end{equation} 
In particular, when $\varepsilon=\delta$ and taking $\theta_\divideontimes\to \frac{\pi}{2}$, then the fundamental gap approaches the constant $3\delta \ell^2.$
\end{lemma}
\begin{proof}
In the assumptions of the present lemma, we have $\lambda_1=\lambda_1^{\ell^2}$ and $\lambda_2=\lambda_1^{4\ell^2}$. Let $h^{(1)}$ and $h^{(2)}$ be the corresponding eigenfunctions of these eigenvalues, respectively; it follows from~\eqref{eq.8} and~\eqref{eq.9} that
\begin{eqnarray*}
(\varphi h^{(1)}_\theta)_\theta+(\lambda_1\csc^2\theta-\ell^2\varphi)h^{(1)} = 0,\\
(\varphi h^{(2)}_\theta)_\theta+(\lambda_2\csc^2\theta-4\ell^2\varphi)h^{(2)} = 0.
\end{eqnarray*}
To obtain the lower estimate, recall that $\csc^2\theta\leq \csc^2\theta_\divideontimes$. The proof is by contradiction. Suppose that $\lambda_2\leq \lambda_1+3\varepsilon \ell^2\sin^2\theta_\divideontimes$, then
\begin{align*}
\lambda_2\csc^2\theta -4\ell^2\varphi\leq \lambda_1\csc^2\theta+3\ell^2\varepsilon \sin^2\theta_\divideontimes \csc^2\theta-4\ell^2\varphi\leq \lambda_1\csc^2\theta -\ell^2\varphi.
\end{align*}
Moreover, observe that $\lambda_1\csc^2\theta+3\ell^2\varepsilon \sin^2\theta_\divideontimes \csc^2\theta-4\ell^2\varphi= \lambda_1\csc^2\theta -\ell^2\varphi$ is equivalent to $ \sin^2\theta= \frac{\varepsilon}{\varphi}\sin^2\theta_\divideontimes$, and as $\frac{\varepsilon}{\varphi}\leq 1$, then $\sin\theta\leq \sin\theta_\divideontimes$, so, $\theta\leq \theta_\divideontimes$ or $\theta\geq \pi-\theta_\divideontimes$, hence, $\theta\leq \theta_0$ or $\theta\geq \theta_1$. Consequently, the inequality is strict in $(\theta_0,\theta_1)$, which allows us to use Theorem~\ref{teosturn} to conclude that $h^{(1)}$ has at least one zero in the interval $(\theta_0,\theta_1)$, which contradicts Theorem~\ref{teosturm1}. The upper estimate is obtained in the same way by using that $\csc^2\theta\geq 1$.
\end{proof}
To obtain an estimate for the fundamental gap of the operator $\dv(\varphi\nabla u)$ in~\eqref{opchengyau} on each domain of the family in~\eqref{OMEGA}, we use variational arguments as in~\cite[Sec.~3]{Lavine} and \cite[Sec.~5]{Bourni}, as well as the Sturm--Liouville result. For it,
we consider the one-parameter family of problems
\begin{eqnarray}\label{eq.18}
(\varphi h_\theta)_\theta+\lambda\csc^2\theta h&=&\mu(s)\varphi(\theta) h\quad\mbox{em}\quad (\theta_0,\theta_1)\\
\nonumber h(\theta_0)=h(\theta_1)&=&0,
\end{eqnarray}
where $h(\theta) = h^s(\theta)$, $\lambda=\lambda(s)$ and $\mu(s)$ is a smooth curve such that $\mu(0)=\ell^2$ and $\mu(1)=4\ell^2$, with $0\leq s\leq 1$. For each $s$, let $\lambda(s)$ be the smallest eigenvalue that is smooth at $s$, and let $h^s(\theta)$ be its first eigenfunction given by Theorem~\ref{teosturm1}, which satisfies $\int_{\theta_0}^{\theta_1}\csc^2\theta (h^s(\theta))^2d\theta=1$ and $h^s(\theta) > 0$ on $(\theta_0,\theta_1)$. Denoting by $\dot{f}$ the derivative with respect to $s$ of a function $f(s)$, we get 
\begin{equation*}
(\varphi(\theta)\dot{h}_\theta)_\theta+\lambda(s)\csc^2\theta\dot{h}-\mu(s)\varphi(\theta)\dot{h}+\dot{\lambda}(s)\csc^2\theta h=\dot{\mu}(s)\varphi(\theta)h\quad\mbox{on}\quad (\theta_0,\theta_1).
\end{equation*}
Multiplying the previous equation by $h$, and by integrating from $\theta_0$ to $\theta_1$, we have
\begin{align}\label{eq.3.25}
\nonumber&\int_{\theta_0}^{\theta_1}\big((\varphi(\theta)\dot{h}_\theta)_\theta+\lambda(s)\csc^2\theta\dot{h}-\mu(s)\varphi(\theta)\dot{h}\big)h d\theta\\
&=\int_{\theta_0}^{\theta_1}\dot{\mu}(s)\varphi(\theta)h^2d\theta -\int_{\theta_0}^{\theta_1}\dot{\lambda}(s)\csc^2\theta h^2d\theta.
\end{align}
Using integration by parts and \eqref{eq.18} to evaluate the term
\begin{equation*}
\int_{\theta_0}^{\theta_1}(\varphi(\theta)\dot{h}_\theta)_\theta hd\theta=\int_{\theta_0}^{\theta_1}(\varphi(\theta) h_\theta)_\theta\dot{h}d\theta=\int_{\theta_0}^{\theta_1}\big(\mu(s)\varphi(\theta)h\dot{h} -\lambda(s)\csc^2\theta h\dot{h}\big)d\theta,
\end{equation*}
equation~\eqref{eq.3.25} reduces to 
\begin{equation*}
\dot{\mu}(s)\int_{\theta_0}^{\theta_1} \varphi(\theta) h^2d\theta=\dot{\lambda}(s)\int_{\theta_0}^{\theta_1} \csc^2\theta h^2d\theta=\dot{\lambda}(s).
\end{equation*}
Note that we can take the curve $\mu(s) = \ell^2+3\ell^2s$, so that
\begin{equation*}
\dot{\lambda}(s)=3\ell^2\int_{\theta_0}^{\theta_1} \varphi (\theta)(h^s(\theta))^2d\theta.
\end{equation*}
Integrating from $0$ to $1$, and as $\lambda(0) =\lambda_1$ and $\lambda(1) = \lambda_2$, we obtain
\begin{equation}\label{eq.20}
\lambda_2-\lambda_1\leq 3\ell^2 \delta\max_{s\in[0,1]}\int_{\theta_0}^{\theta_1} (h^s(\theta))^2d\theta.
\end{equation}
Now, it is enough to estimate the right-hand side of~\eqref{eq.20}.
\begin{proposition}\label{prop5.1} It is valid that
\begin{equation}\label{eq.21}
\max_{s\in [0,1]} \int_{\theta_0}^{\theta_1} (h^s(\theta))^2d\theta< 1.
\end{equation}
\end{proposition}
\begin{proof}
First, note that $\int_{\theta_0}^{\theta_1} (h^s(\theta))^2d\theta\leq \int_{\theta_0}^{\theta_1} \csc^2\theta(h^s(\theta))^2d\theta = 1$. We will show that $\int_{\theta_0}^{\theta_1} (h^s(\theta))^2d\theta<1$ for all $s\in[0,1]$. For this purpose, we consider the angle $\alpha$ satisfying
\begin{equation*}
\theta_0< \alpha\ < \frac{\pi}{2}.
\end{equation*}
The proof is by contradiction. Assume the equality
\begin{eqnarray*}
\int_{\theta_0}^{\theta_1}\csc^2\theta (h^s(\theta))^2d\theta=&1&= \int_{\theta_0}^{\theta_1}(h^s(\theta))^2d\theta,
\end{eqnarray*}
for some $s\in[0,1]$. Recall that $\theta_1>\frac{\pi}{2}$, so that we can write
\begin{eqnarray*}
\int_{\theta_0}^{\alpha}\csc^2\theta (h^s(\theta))^2d\theta+\int_{\alpha}^{\theta_1}\csc^2\theta (h^s(\theta))^2 d\theta&=& \int_{\theta_0}^{\alpha}(h^s(\theta))^2d\theta+\int_{\alpha}^{\theta_1}(h^s(\theta))^2d\theta.
\end{eqnarray*}
The choice of $\alpha$ implies $\csc^2\theta\geq \csc^2\alpha$ on $(\theta_0,\alpha],$ moreover, $1\leq\csc^2\theta$, then
\begin{eqnarray*}
\csc^2\alpha\int_{\theta_0}^{\alpha} (h^s(\theta))^2d\theta+\int_{\alpha}^{\theta_1}(h^s(\theta))^2 d\theta&\leq& \int_{\theta_0}^{\alpha}(h^s(\theta))^2d\theta+\int_{\alpha}^{\theta_1}(h^s(\theta))^2d\theta.
\end{eqnarray*}
Hence,
\begin{eqnarray*}
(\csc^2\alpha-1)\int_{\theta_0}^{\alpha} (h^s(\theta))^2d\theta&\leq&0,
\end{eqnarray*}
which is a contradiction because $\csc^2\alpha>1$. So, $\int_{\theta_0}^{\theta_1}(h^s(\theta))^2d\theta<1$ for all $s\in[0,1]$, which is enough to obtain the result of the proposition.
\end{proof}

\begin{theorem}\label{teo1}
The fundamental gap for the operator $\dv(\varphi\nabla u)$, with $\varepsilon\leq\varphi\leq\delta$ and $\varphi_r=0$, on each set of the family of convex domains $\lim_{\theta_\divideontimes\to \frac{\pi}{2}}\Omega_{\ell,\theta_0,\theta_1}$, where $\Omega_{\ell,\theta_0,\theta_1}$ is defined by~\eqref{OMEGA}, satisfies
\begin{equation*}
(\lambda_2-\lambda_1)D^2<3\pi^2\delta,
\end{equation*}
where $D=\displaystyle\lim_{\theta_\divideontimes\to \frac{\pi}{2}}D_{\ell,\theta_0,\theta_1}$ is the diameter of each set of the family $\displaystyle\lim_{\theta_\divideontimes\to \frac{\pi}{2}}\Omega_{\ell,\theta_0,\theta_1}.$
\end{theorem}

\begin{proof} By~\eqref{eq.20} and Proposition~\ref{prop5.1}, we have that
\begin{equation*}
    \lambda_2-\lambda_1 < 3\ell^2\delta.
\end{equation*}
On the other hand, from Lemma~\ref{lema4.1} 
\begin{equation*}
3\ell^2\delta = \lim_{\theta_{*}\to\frac{\pi}{2}}\frac{3\pi^2\delta}{D^2_{\ell,\theta_0,\theta_1}}
\end{equation*}
which is enough to complete the result of the theorem.
\end{proof}

From Theorem~\ref{teo1} we prove the next result for the operator $\Lu$ defined as in~\eqref{defLintro}, with $T=\varphi I$ and $\varphi_r=0.$

\begin{corollary}\label{Cor6-Sec6}
Consider the operator $\dv(\varphi\nabla u)$, with $\varepsilon\leq\varphi\leq\delta$ and $\varphi_r=0$, on each set of the family of convex domains $\lim_{\theta_\divideontimes\to \frac{\pi}{2}}\Omega_{\ell,\theta_0,\theta_1}$, where $\Omega_{\ell,\theta_0,\theta_1}$ is defined by~\eqref{OMEGA}. There exists a drifting function $\eta$ such that the fundamental gap for $\dv(\varphi\nabla u)$ remains invariant under a first-order perturbation of this operator. More precisely, for the operator $\Lu u =\dv(\varphi\nabla u)-\langle\nabla\eta,\varphi\nabla u\rangle$ it is true that
\begin{equation*}
(\lambda_2-\lambda_1)D^2<3\pi^2\delta,
\end{equation*}
where $D=\displaystyle\lim_{\theta_\divideontimes\to \frac{\pi}{2}}D_{\ell,\theta_0,\theta_1}$ is the diameter of each set of the family $\displaystyle\lim_{\theta_\divideontimes\to \frac{\pi}{2}}\Omega_{\ell,\theta_0,\theta_1}.$
\end{corollary}
\begin{proof}
We claim that any solution $\eta$ of the equation
\begin{equation}\label{eq.eta.1}
    2\dv(\varphi \nabla \eta)-\g{\nabla\eta}{\varphi \nabla \eta}=0
\end{equation}
can be taken as a drifting function so that the fundamental gap for the operator $\Lu =\dv(\varphi\nabla \cdot)-\langle\nabla\eta,\varphi\nabla \cdot\rangle$, with $\varepsilon\leq\varphi\leq\delta$ and $\varphi_r=0$, on each set of the family of convex domains $\lim_{\theta_\divideontimes\to \frac{\pi}{2}}\Omega_{\ell,\theta_0,\theta_1}$, where $\Omega_{\ell,\theta_0,\theta_1}$ is defined by~\eqref{OMEGA}, satisfies the required estimate. Indeed, we take $u$ such that $-\lambda u=\dv(\varphi\nabla u)$, and the change of variable $u=e^{-\frac{\eta}{2}}v$ together with the property of divergence of a vector field, and then we obtain
\begin{equation*}
\dv(\varphi\nabla u)=e^{-\frac{\eta}{2}}\Big( \dv(\varphi\nabla v) -\g{\nabla v}{\varphi \nabla\eta}\Big)-\frac{e^{-\frac{\eta}{2}}v}{4}\Big(2\dv(\varphi\nabla \eta)-\g{\nabla \eta}{\varphi\nabla\eta}\Big).
\end{equation*}
If $\eta$ solves \eqref{eq.eta.1}, then $-\lambda e^{-\frac{\eta}{2}}v=-\lambda u=\dv(\varphi\nabla u) =e^{-\frac{\eta}{2}}\Lu v$. Hence, $\Lu v=-\lambda v$, and thus the estimate of the fundamental gap for $\Lu$ follows from Theorem~\ref{teo1}.
\end{proof}

The reader can readily see that it is easy to obtain solutions to equation~\eqref{eq.eta.1}. Next, we present a simple example.

\begin{example}
Let $\varphi:(1,e^{\frac{\pi}{\ell}})\times(\pi/3,2\pi/3)\subset\mathbb{H}^2\to\mathbb{R}$ be a radially constant function, defined by $\varphi(r,\theta)=\sin\theta$. Note that $\varphi$ and $\theta$ satisfy the conditions in \eqref{OMEGA}, for $\theta_0=\pi/3,$ $\theta_1=2\pi/3,$ $\theta_*=\pi/3,$ $\varepsilon=\sqrt{3}/2,$ $\delta=1$ and $\ell$ to be appropriately chosen. Under these conditions, the functions $\eta(r,\theta)=-2\ln(1-\ln\tan\frac{\theta}{2})$ and $\eta(r,\theta)=-2\ln(\pi-\ell\ln r)$ solve equation~\eqref{eq.eta.1}. 
\end{example}
Indeed, we need to prove that $2\dv(\varphi \nabla \eta)-\g{\nabla\eta}{\varphi \nabla \eta}=0$. For it, we can proceed as in \eqref{PI-quadrado} and \eqref{opchengyau}, to get the next expression, which is valid for any $\varphi$ radially constant and for any $\eta,$
\begin{align*}
2\dv(\varphi \nabla \eta)-\g{\nabla\eta}{\varphi \nabla \eta}&= 2(\varphi(r^2\sin^2\theta\eta_{rr}+r\sin^2\theta\eta_r)+\varphi\sin^2\theta\eta_{\theta\theta}+\sin^2\theta \varphi_\theta\eta_\theta)\\
&\quad-\varphi(r^2\sin^2\theta \eta_r\eta_r+\sin^2\theta \eta_\theta \eta_\theta).
\end{align*}
For $\eta(r,\theta)=-2\ln(1-\ln\tan\frac{\theta}{2})$, we have
\begin{eqnarray*}
\eta_r=0,\,\,\, \eta_\theta=\frac{2\csc\theta}{1-\ln\tan\frac{\theta}{2}} \,\,\,\mbox{and}\,\,\, \eta_{\theta\theta}=\frac{-2\csc\theta\cot\theta}{1-\ln\tan\frac{\theta}{2}}+\frac{2\csc^2\theta}{(1-\ln\tan\frac{\theta}{2})^2}.
\end{eqnarray*}
Hence, 
\begin{equation*}
2\dv(\varphi \nabla \eta)-\g{\nabla\eta}{\varphi \nabla \eta}=2\varphi\sin^2\theta \eta_{\theta\theta}+2\sin^2\theta \varphi_\theta \eta_\theta-\varphi\sin^2\theta \eta_\theta^2=0.
\end{equation*}
For $\eta(r,\theta)=-2\ln(\pi-\ell\ln r)$, the computation is simple as in the previous case.

\begin{remark} For the Laplacian case with Dirichlet boundary conditions on convex domains in $\mathbb{H}^n$, $n \geq 2$, Bourni et al. proved that the product of the fundamental gap with the square of the diameter can be arbitrarily small for domains of any diameter (see \cite[Theorem~1.1]{Bourni1}). Then, we can say that they also proved that the drifted Laplacian has the same property when its drifting function solves $2\Delta\eta-|\nabla\eta|^2=0.$
\end{remark}

\section{Concluding remarks}\label{ConcludeRemars}
Here, we give some applications of Theorem~\ref{thmA3-c3} in a more general context.
\begin{corollary}\label{Cor1ConcRem}
Let $\lambda_1$ be the first eigenvalue of the drifted Cheng--Yau operator with a drifting function $\eta$ on a bounded domain $\Omega\subset M^n$ with the Dirichlet boundary condition. Fix an origin $o\in M^n\backslash\overline{\Omega}$, and let $r(x)$ be the distance function from $o$. If $T$ is radially parallel and it has $\partial_r$ as an eigenvector, then:
\begin{enumerate}
\item For $a(n,\varepsilon,\delta)\leq0$,
\begin{equation*}
\lambda_1(\Omega) \geq \frac{\varepsilon}{4\delta^2}\Big[(n-1)^2\varepsilon^2\kappa^2_2
-2(n-1)(\delta^2\kappa^2_1-\varepsilon^2\kappa^2_2) - 2C_0(n-1)\big(\kappa_1+\frac{1}{d}\big) - C_1\Big].
\end{equation*}
\item For $a(n,\varepsilon,\delta)>0$,
\begin{equation*}
\lambda_1(\Omega)\geq \frac{\varepsilon}{4\delta^2}\Big[(n-1)^2\varepsilon^2\kappa^2_2
-2(n-1)(\delta^2\kappa^2_1-\varepsilon^2\kappa^2_2) - 2C_0(n-1)\big(\kappa_1+\frac{1}{d}\big) - C_1 -\frac{a(n,\varepsilon,\delta)}{d^2}\Big],
\end{equation*}
\end{enumerate}
where the constants $d$, $C_0$ and $C_1$ are as in Theorem~\ref{thmA3-c3}.
\end{corollary}
\begin{proof}
For simplicity and further reference, we define the auxiliary sequences:
\begin{enumerate}
\item\label{vi1-Cor} For $a(n,\varepsilon,\delta)\leq0$,
\begin{equation*}
\upsilon_i:=\frac{4\delta^2}{\varepsilon}\lambda_i -(n-1)^2\varepsilon^2\kappa^2_2
+2(n-1)(\delta^2\kappa^2_1-\varepsilon^2\kappa^2_2)+ 2C_0(n-1)\big(\kappa_1+\frac{1}{d}\big)+C_1.
\end{equation*}
\item\label{vi2-Cor} For $a(n,\varepsilon,\delta)>0$,
\begin{equation*}
\upsilon_i:=\frac{4\delta^2}{\varepsilon}\lambda_i -(n-1)^2\varepsilon^2\kappa^2_2
+2(n-1)(\delta^2\kappa^2_1-\varepsilon^2\kappa^2_2)+ 2C_0(n-1)\big(\kappa_1+\frac{1}{d}\big)+C_1+\frac{a(n,\varepsilon,\delta)}{d^2},
\end{equation*}
\end{enumerate}
where the constants $d$, $C_0$ and $C_1$ are as in Theorem~\ref{thmA3-c3}. From \eqref{Eq1-Prop1GM} and the inequalities in Theorem~\ref{thmA3-c3}, we have 
\begin{equation*}
0\leq\Vert u_i\mathscr{L}r+2T(\partial_r,\nabla u_i)\Vert_{L^2(\Omega,\dm)}^2\leq \upsilon_i.
\end{equation*}
Take $i=1$ to obtain the result of the corollary.
\end{proof}

Now, we prove some estimates of eigenvalues of the drifted Cheng--Yau operator. In particular, we obtain the corresponding estimates for the Cheng--Yau operator, the drifted Laplacian and the Laplacian. For more in-depth information on these estimates, we refer the reader to~\cite{GomesMiranda}.

\begin{theorem}\label{StandarEstimate-CYO}
Let $\lambda_i$ be the $i$-th eigenvalue of the drifted Cheng--Yau operator with a drifting function $\eta$ on a bounded domain $\Omega\subset M^n$ with the Dirichlet boundary condition. Fix an origin $o\in M^n\backslash\overline{\Omega}$, and let $r(x)$ be the distance function from $o$. Suppose $T$ is radially parallel and it has $\partial_r$ as an eigenvector. Then, the sequence of eigenvalues $(\lambda_i)$ satisfies the following estimates in terms of the auxiliary sequences $(v_i)$ defined in  Corollary~\ref{Cor1ConcRem}:
\begin{equation}\label{2-Aux-Num}
\upsilon_{k+1}\leq \Big(1+\frac{4\delta^2}{\varepsilon^2}\Big)k^\frac{2\delta^2}{\varepsilon^2}\upsilon_{1}.
\end{equation}
\begin{equation}\label{3-Aux-Num}
\upsilon_{k+1}\leq\frac{1}{k}\Big(1+\frac{4\delta^2}{\varepsilon^2}\Big)\sum_{i=1}^k\upsilon_i.
\end{equation}
\begin{equation}\label{4-Aux-Num}
\upsilon_{k+1}\leq \Big(1+\frac{2\delta^2}{\varepsilon^2}\Big)\frac{1}{k}\sum_{i=1}^k\upsilon_{i}+\Big[\Big(\frac{2\delta^2}{\varepsilon^2}\frac{1}{k}\sum_{i=1}^{k}\upsilon_i\Big)^2-\Big(1+\frac{4\delta^2}{\varepsilon^2}\Big)\frac{1}{k}\sum_{j=1}^k\Big(\upsilon_{j}-\frac{1}{k}\sum_{i=1}^{k}\upsilon_{i}\Big)^2\Big]^\frac{1}{2}.
\end{equation}
\begin{equation}\label{5-Aux-Num}
\upsilon_{k+1}-\upsilon_{k}\leq 2\Big[\Big(\frac{2\delta^2}{\varepsilon^2}\frac{1}{k}\sum_{i=1}^{k}\upsilon_{i}\Big)^2-\Big(1+\frac{4\delta^2}{\varepsilon^2}\Big)\frac{1}{k}\sum_{j=1}^k\Big(\upsilon_{j}-\frac{1}{k}\sum_{i=1}^{k}\upsilon_{i}\Big)^2\Big]^\frac{1}{2}.
\end{equation}
\end{theorem}

\begin{proof}
Note that we can write each case of the sequence $(\upsilon_i)$ as follows
\begin{equation}\label{Aux-Number}
\upsilon_i=\frac{4\delta^2}{\varepsilon}\lambda_i + C,
\end{equation}
for an appropriate constant $C$. With this simplified notation, inequalities of Theorem~\ref{thmA3-c3} become
\begin{equation*}
\sum_{i=1}^k(\lambda_{k+1}-\lambda_i)^2 \leq\frac{1}{\varepsilon}\sum_{i=1}^k(\lambda_{k+1} -\lambda_i)\Big(\frac{4\delta^2}{\varepsilon}\lambda_i +C \Big).
\end{equation*}
By \eqref{Aux-Number}, we get 
\begin{equation*}
\lambda_{k+1}-\lambda_i= \frac{\varepsilon}{4\delta^2}(\upsilon_{k+1}-\upsilon_i).
\end{equation*}
Then,
\begin{equation}\label{1-Aux-Num}
\sum_{i=1}^k(\upsilon_{k+1}-\upsilon_i)^2 \leq \frac{4\delta^2}{\varepsilon^2} \sum_{i=1}^k (\upsilon_{k+1}-\upsilon_i)\upsilon_i.
\end{equation}
Besides, note that $\upsilon_1\leq\upsilon_2\leq\cdots\to\infty$, since  $\lambda_{1}\leq\lambda_{2}\leq\cdots\to\infty$. Hence, each case of the sequence $(\upsilon_i)$ satisfies the assumptions as in Cheng and Yang~\cite{Cheng-Yang,Cheng-Yang-III}, from which we get \eqref{2-Aux-Num}. 

By using \eqref{1-Aux-Num}, we follow the same steps as in Theorem~3 of \cite{GomesMiranda} to prove the estimates in \eqref{3-Aux-Num}, \eqref{4-Aux-Num} and \eqref{5-Aux-Num}.
\end{proof}

Finally, we make some important observations of independent interest concerning the more general expression for $\mathscr{L}$.
For this, let $N(\lambda)$ be the eigenvalue distribution function given by the number of eigenvalues $\lambda_k$ smaller than a given $\lambda$. We observe that the principal symbol of $\mathscr{L}$ is given by $T_x(\xi,\xi)$ for $x\in M^n$ and $\xi\in T_xM$, see Eqs.~\eqref{quadrado} and \eqref{L-Box-div}. It determines the first term in the asymptotics of the eigenvalue distribution function, as follows:
\begin{equation}\label{N}
	N(\lambda)=c_0\lambda^{\frac{n}{2}}+O(\lambda^{\frac{n-1}{2}})\quad\hbox{as}\quad \lambda\to\infty
\end{equation}
where 
$$c_0=(2\pi)^{-n}\mathrm{vol}\big\{(x,\xi):T_x(\xi,\xi)\leq 1\big\}=(2\pi)^{-n}\int_{\{T_x(\xi,\xi)\leq 1\}}dxd\xi.$$
For more details, see \cite{LH,VIv,RSee,DGV}. From \eqref{N}, we obtain
\begin{equation}\label{O1}
\lambda_k=c_0^{-\frac{2}{n}}k^{\frac{2}{n}}+O(1)\quad\hbox{as}\quad k\to\infty.
\end{equation}
We now deduce two identities that have been used as necessary tools on eigenvalue problems in particular cases of $\mathscr{L}$, namely:
\begin{equation}\label{limlambda}
\lim_{k\rightarrow\infty}\frac{\frac{1}{k}\sum_{i=1}^k\lambda_i}{k^{\frac{2}{n}}}=\frac{n}{n+2}c_0^{-\frac{2}{n}}
\quad\hbox{and}\quad 
\lim_{k\rightarrow\infty}\frac{\frac{1}{k}\sum_{i=1}^k\lambda_i^2}{k^{\frac{4}{n}}}=\frac{n}{n+4}c_0^{-\frac{4}{n}}.
\end{equation}
Indeed, from \eqref{O1} we get 
\begin{equation}\label{Aux-O1}
\frac{\frac{1}{k}\sum_{i=1}^k\lambda_i}{k^\frac{2}{n}}=c_0^{-\frac{2}{n}}\sum_{i=1}^k\Big(\frac{i}{k}\Big)^\frac{2}{n}\frac{1}{k}+\frac{O(1)}{k^\frac{2}{n}}
\end{equation}
Consider $f(t)=t^\frac{2}{n}$ on $[0,1]$, and the partition $0<\frac{1}{k}<\frac{2}{k}<\cdots<\frac{i}{k}<\cdots<1$ of $[0,1]$ so that
\begin{equation}\label{Aux-O11}
\lim_{k\to\infty}\sum_{i=1}^kf\Big(\frac{i}{k}\Big)\frac{1}{k}=	\lim_{k\to\infty}\sum_{i=1}^k\Big(\frac{i}{k}\Big)^\frac{2}{n}\frac{1}{k}=\int_0^1t^\frac{2}{n}dt=\frac{n}{n+2}.
\end{equation}
Thus, the first identity in \eqref{limlambda} follows from \eqref{Aux-O1}. 

For the second identity in \eqref{limlambda}, we begin by noting that 
\begin{equation*}
\lambda_k^2=c_0^{-\frac{4}{n}}k^{\frac{4}{n}}+2c_0^{-\frac{2}{n}}k^{\frac{2}{n}}O(1)+O(1)^2\quad\hbox{as}\quad k\to\infty.
\end{equation*}
So,
\begin{equation}\label{Aux-O1^2}
\frac{\frac{1}{k}\sum_{i=1}^k\lambda_i^2}{k^\frac{4}{n}}=c_0^{-\frac{4}{n}}\sum_{i=1}^k\Big(\frac{i}{k}\Big)^\frac{4}{n}\frac{1}{k}+2O(1)c_0^{-\frac{2}{n}}\frac{1}{k^{\frac{2}{n}}}\sum_{i=1}^k\Big(\frac{i}{k}\Big)^\frac{2}{n}\frac{1}{k} + \frac{O(1)^2}{k^\frac{4}{n}}.
\end{equation}
As in the first case, we use the function $t\mapsto t^\frac{4}{n}$ to get
\begin{equation*}
\lim_{k\to\infty}\sum_{i=1}^k\Big(\frac{i}{k}\Big)^\frac{4}{n}\frac{1}{k}=\int_0^1t^\frac{4}{n}dt=\frac{n}{n+4}.
\end{equation*}
Hence, the second identity in \eqref{limlambda} follows from \eqref{Aux-O11} and \eqref{Aux-O1^2}.

\section{Acknowledgements}
We would like to thank the reviewer for their careful reading of the manuscript and for the valuable comments and suggestions, which helped us improve the clarity and presentation of the manuscript and better highlight the novelty of our results.

\section{\bf Statements and declarations}

\subsection{\bf Competing interests} The authors declare that they have no competing interests relevant to the content of this article.

\subsection{\bf Author contributions}
All authors wrote the main manuscript text and reviewed the manuscript.

\subsection{\bf Funding} José N.V. Gomes has been partially supported by the Conselho Nacional de Desenvolvimento Científico e Tecnológico (CNPq), Grant number 309825/2026-1, and by the Fundação de Amparo à Pesquisa do Estado de São Paulo (FAPESP), Grant numbers 2024/00923-6, 2023/11126-7, and 2022/16097-2.


\begin{thebibliography}{99}
\bibitem{AGD} H. Alencar, G. S. Neto and D. Zhou, Eigenvalue estimates for a class of elliptic differential operators on compact manifolds, Bull. Braz. Math. Soc. (N.S.) 46 (3) (2015), 491-514.
\bibitem{Andrews} B. Andrews and J. Clutterbuck, Proof of the fundamental gap conjecture, J. Amer. Math. Soc. 24 (3) (2011), 899-916.
\bibitem{AB} M. S. Ashbaugh and R. Benguria, Optimal lower bound for the gap between the first two eigenvalues of one-dimensional Schrödinger operators with symmetric single-well potentials, Proc. Amer. Math. Soc. 105 (2) (1989), 419-424.
\bibitem{MvdB} M. van den Berg, On condensation in the free-boson gas and the spectrum of the Laplacian, J. Statist. Phys. 31 (3) (1983), 623-637.
\bibitem{Bishop-O'Neil} R. L. Bishop and B. O’Neill, Manifolds of negative curvature, Trans. Amer. Math. Soc. 145 (1969), 1-49.
\bibitem{Bourni}T. Bourni, J. Clutterbuck, X. H. Nguyen, A. Stancu, G. Wei and V. M. Wheeler, Explicit Fundamental gap estimates for some convex domains in $\mathbb{H}^2$. Math. Res. Lett. 28 (5) (2021), 1319-1336.
\bibitem{Bourni1} T. Bourni, J. Clutterbuck, X. H. Nguyen, A. Stancu, G. Wei and V. M. Wheeler, The vanishing of the fundamental gap of convex domains in $\mathbb{H}^n,$ Ann. Henri Poincaré 23 (2022), 595-614.
\bibitem{Chavel} I. Chavel, Eigenvalues in Riemannian Geometry, Academic Press, New York, 1984.
\bibitem{CHEN-TAO ZHENG} D. Chen, T. Zheng and M. Lu, Eigenvalue estimates on domains in complete noncompact Riemannian manifolds, Pacific J. Math. 255 (1) (2012), 41-54.
\bibitem{S.Y.Cheng-I} S. Y. Cheng, Eigenfunctions and eigenvalues of the Laplacian, Am. Math. Soc. Proc. Symp. Pure Math. 27 (1975) Part II, 185-193.
\bibitem{Cheng-Yang} Q. M. Cheng and H. C. Yang, Bounds on eigenvalues of Dirichlet Laplacian, Math. Ann. 337 (2007), 159-175.
\bibitem{Cheng-Yang-III} Q. M. Cheng and H. C. Yang, Estimates for eigenvalues on Riemannian Manifolds, J. Differential Equations 247 (2009), 2270-2281.
\bibitem{Cheng-Yau} S. Y. Cheng and S.-T. Yau, Hypersurfaces with constant scalar curvature, Math. Ann. 225 (1977), 195-204.
\bibitem{Dai} X. Dai, S. Seto and G. Wei, Fundamental gap estimate for convex domains on sphere - the case $n = 2$, Comm. Anal. Geom. 29 (5) (2021), 1095-1125.
\bibitem{yau}  R. E. Greene and S.-T. Yau, Differential Geometry: Partial Differential Equations on Manifolds, Proceedings of Symposia in Pure Mathematics, Volume 54 (1993), Part 1. 
\bibitem{RJS} R. Gicquaud, D. Ji and Y. Shi, On the asymptotic behavior of Einstein manifolds with an integral bound on the Weyl curvature, Comm. Anal. Geom.  21 (5) (2013), 1081-1113.
\bibitem{GomesMiranda} J. N. V. Gomes and J. F. R. Miranda, Eigenvalue estimates for a class of elliptic differential operators in divergence form, Nonlinear Anal. 176 (2018), 1-19.
\bibitem{He} C. He, G. Wei and Q. S. Zhang, Fundamental gap of convex domains in the spheres, Amer. J. Math. 142 (4) (2020), 1161-1192.
\bibitem{AHenrot} A. Henrot, Extremum problems for eigenvalues of elliptic operators. Springer Science \& Business Media, 2006.
\bibitem{LH} L. Hörmander, The spectral function of an elliptic operator, Acta Math. 121 (1968), 193-218.
\bibitem{VIv} V. Ivrii, Precise spectral asymptotics for elliptic operators acting in fiberings over manifolds with boundary, Lecture Notes in Math., vol. 1100, Springer-Verlag, Berlin and New York, 1984.
\bibitem{Kielhofer}H. Kielhöfer, Calculus of Variations: An introduction to the one-dimensional theory with examples and exercises, Springer, Rimsting, Bayern, Germany, 2018.
\bibitem{Lavine} R. Lavine, The eigenvalue gap for one-dimensional convex potentials, Proc. Amer. Math. Soc. 121 (3) (1994), 815-821.
\bibitem{Marrocos-Gomes} M. A. M. Marrocos and J. N. V. Gomes, Generic Spectrum of Warped Products and G-Manifolds, J. Geom. Anal. 29 (2019), 3124-3134.
\bibitem{McKean} H. P. McKean, An upper bound to the spectrum of $\Delta$ on a manifold of negative curvature, J. Differential Geom. 4 (3) (1970), 359-366.
\bibitem{Navarro} J. Navarro, On second-order, divergence-free tensors, J. Math. Phys. 55 (2014) 062501. 
\bibitem{Petersen} P. Petersen, Riemannian Geometry, Third Edition, Springer, Los Angeles, USA, 2016.
\bibitem{PinskyI} M. A. Pinsky, The spectrum of the Laplacian on a manifold of negative curvature. I, J. Differential Geom. 13 (1978), 87-91.
\bibitem{RSee} R. Seeley, An estimate near the boundary for the spectral function of the Laplace operator, Amer. J. Math. 102 (1980), 869-902.
\bibitem{Serre} D. Serre, Divergence-free positive symmetric tensors and ﬂuid dynamic, Ann. Inst. H. Poincaré Anal. Non Linéaire 35 (5) (2018), 1209-1234.
\bibitem{Seto} S. Seto, L. Wang and G. Wei, Sharp fundamental gap estimate on convex domains of sphere, J. Differential Geom. 112 (2) (2019), 347-389.
\bibitem{Jorge} J. Sotomayor, Lições de equações diferenciais ordinárias, Projeto Euclides. Impa. 1979.
\bibitem{DGV} D. G. Vasil'ev, Asymptotics of the spectrum of a boundary value problem, Trudy Moskov. Mat. Obshch. 49 (1986), 167-237; English transl. in Trans. Moscow Math. Soc. (1987), 173-245.
\bibitem{Y1} S.-T. Yau, Nonlinear analysis in geometry, Monographies de L'Enseignement Mathématique, vol. 33, L'Enseignement Mathématique, Geneva, 1986. Série des Conférences de l’Union Mathématique Internationale, 8. MR865650 (88e:53001).
\bibitem{zeld} S. Zelditch, On the generic spectrum of a Riemannian cover, Ann. Inst. Fourier 40 (2) (1990), 407-442.
\end{thebibliography}
\end{document}